\documentclass[a4paper,10pt]{amsart}
\usepackage{amsfonts,amsmath,amssymb,amsthm,graphicx,mathrsfs,tikz,xcolor,hyperref}
\usepackage{chngcntr}
\usepackage[a4paper, total={6in, 10in}]{geometry}
\hypersetup{
	colorlinks   = true, 
	urlcolor     = blue, 
	linkcolor    = green, 
	citecolor   = red 
}
\tikzset{Bullet/.style={fill=black,draw,color=#1,circle,minimum size=2pt,scale=0.45}}
\usetikzlibrary{shapes.geometric}
\usetikzlibrary{positioning, tikzmark}
\usepackage[all]{xy}
\usepackage[utf8]{inputenc}
\usepackage{multirow,array}
\usepackage{tabularx}
\newcolumntype{P}[1]{>{\centering\arraybackslash}p{#1}}
\theoremstyle{plain}

\newtheorem{thm}{Theorem}[section]
\newtheorem{lem}[thm]{Lemma}

\theoremstyle{definition}
\newtheorem{defn}[thm]{Definition}
\newtheorem{obv}[thm]{Observation}
\newtheorem{rem}[thm]{Remark}
\newtheorem{qs}[thm]{Question}
\newtheorem{cor}[thm]{Corollary}
\newtheorem{exm}[thm]{Example}

\numberwithin{equation}{section}

\begin{document}
	\Large{\title{Some familiar graphs on the rings of measurable functions}}
		\author[P. Nandi]{Pratip Nandi}
		\address{Department of Pure Mathematics, University of Calcutta, 35, Ballygunge Circular Road, Kolkata 700019, West Bengal, India}
		\email{pratipnandi10@gmail.com}
		\author[A. Deb Ray]{Atasi Deb Ray}
		\address{Department of Pure Mathematics, University of Calcutta, 35, Ballygunge Circular Road, Kolkata 700019, West Bengal, India}
		\email{debrayatasi@gmail.com}
		\author[S. K. Acharyya]{Sudip Kumar Acharyya}
		\address{Department of Pure Mathematics, University of Calcutta, 35, Ballygunge Circular Road, Kolkata 700019, West Bengal, India}
		\email{sdpacharyya@gmail.com}

		\thanks{The first author thanks the CSIR, New Delhi – 110001, India, for financial support.}
		
		\begin{abstract}
			 In this paper, replacing `equality' by 'equality almost everywhere' we modify several terms associated with the ring of measurable functions defined on a measure space $(X, \mathcal{A}, \mu)$ and thereby study the graph theoretic features of the modified comaximal graph, annihilator graph and the weakly zero-divisor graph of the said ring. The study reveals a structural analogy between the modified versions of the comaximal and the zero-divisor graphs, which prompted us to investigate whether these two graphs are isomorphic. Introducing a quotient-like concept, we find certain subgraphs of the comaximal graph and the zero-divisor graph of $\mathcal{M}(X, \mathcal{A})$ and show that these two subgraphs are always isomorphic. Choosing $\mu$ as a counting measure, we prove that even if these two induced graphs are isomorphic, the parent graphs may not be so. However, in case of Lebesgue measure space on $\mathbb{R}$, we establish that the comaximal and the zero-divisor graphs are isomorphic. Observing that both of the comaximal and the zero-divisor graphs of the ring $\mathcal{M}(X, \mathcal{A})$ are subgraphs of the annihilator graph of the said ring, we find equivalent conditions for their equalities in terms of the partitioning of $X$ into two atoms. Moreover, the non-atomicity of the underlying measure space $X$ is characterized through graph theoretic phenomena of the comaximal and the annihilator graph of  $\mathcal{M}(X, \mathcal{A})$.
		\end{abstract}
	\maketitle
	\section{Introduction}
	Our starting point is a triplet $(X,\mathcal{A},\mu)$, where $X$ is a non-empty set, $\mathcal{A}$, a $\sigma$-algebra of subsets of $X$ and $\mu:\mathcal{A}\to[0,\infty]$ is a measure. Such a triplet is often called a measure space and members of $\mathcal{A}$ are called $\mathcal{A}$-measurable sets or simply measurable sets, if there is no chance of confusion. To make matters non-trivial, we further assume that $\mu(X)>0$. A function $f:X\to\mathbb{R}$ is called measurable or $\mathcal{A}$-measurable if for any $r\in\mathbb{R}$, $\{x\in X:f(x)<r\}$ is a measurable set. The family $\mathcal{M}(X,\mathcal{A})$ of all measurable functions is a commutative lattice ordered ring with identity if the relevant operations are defined pointwise on $X$. Rings of measurable functions are being investigated in the recent times. The articles \cite{Acharyya2020}, \cite{Acharyya2020.1}, \cite{Acharyya2021} are referred to in this connection. Study on several aspects of different well-known graph structures associated with the ring $\mathcal{M}(X,\mathcal{A})$ sounds quite fascinating. In \cite{Hejazipour2022}, the authors take up the study on some relevant properties linked to the zero-divisor graph of the ring $\mathcal{M}(X,\mathcal{A})$, where the zero-divisor graph has been redefined via the measure $\mu$. Incidentally, only one paper \cite{Hejazipour2022} appears in the literature, addressing problems of this kind. In that paper, the authors determined conditions under which this graph becomes triangulated and also realise that it is never hypertriangulated. Being inspired by these findings, in the present article we endeavour to study the comaximal graph, the annihilator graph and the weakly zero-divisor graph of the ring $\mathcal{M}(X,\mathcal{A})$, redefining each of the graphs via the measure $\mu$. These three graphs are generalized versions of the comaximal graph, annihilator graph and the weakly zero-divisor respectively on $\mathcal{M}(X,\mathcal{A})$.\\ 
	To make the article self-contained and reader friendly, we recall in section \ref{Sec2} a few technical terms related to the graphs in general. A few of the ring theoretic terms related to the parent ring $\mathcal{M}(X,\mathcal{A})$ demand natural modifications, when `equality' is replaced by `equality almost everywhere' in the presence of a measure $\mu$ on $(X,\mathcal{A})$. In the last part of this section, we incorporate these changes thereby preparing our measure theoretic preliminaries, necessary for any study on the contemplated graphs on the ring $\mathcal{M}(X,\mathcal{A})$ mentioned above in the subsequent sections.\\ 
	In section \ref{Sec3}, we take up the actual problem of studying the comaximal graph $\Gamma'_2(\mathcal{M}(X,\mathcal{A}))$ of the ring $\mathcal{M}(X,\mathcal{A})$. We compute the familiar graph parameters, viz. the eccentricities, diameter and girth of this graph. This leads to a necessary and sufficient condition for the graph $\Gamma'_2(\mathcal{M}(X,\mathcal{A}))$ to be a complete bipartite one [Theorem~\ref{Th3.13}]. Incidentally it is also established that $\Gamma'_2(\mathcal{M}(X,\mathcal{A}))$ is a triangulated graph if and only if the measure `$\mu$' is non-atomic in the sense that each measurable subset of $X$ with positive measure is expressible as a disjoint union of two measurable sets each with positive measure. Thus in particular if $X=\mathbb{R}$, $\mu\equiv$ Lebesgue or Borel measure, then $\Gamma'_2(\mathcal{M}(X,\mathcal{A}))$ is triangulated. It is further realized that $\Gamma'_2(\mathcal{M}(X,\mathcal{A}))$ is never hypertriangulated [Corollary~\ref{NotHyp}]. Thus outwardly a lot of similarities are observed between the properties of the zero-divisor graph $\Gamma(\mathcal{M}(X,\mathcal{A}))$ as established in \cite{Hejazipour2022} and the corresponding properties of $\Gamma'_2(\mathcal{M}(X,\mathcal{A}))$ as realized in the present paper. It is therefore natural to ask, whether these two graphs, viz. $\Gamma(\mathcal{M}(X,\mathcal{A}))$ and $\Gamma'_2(\mathcal{M}(X,\mathcal{A}))$, become isomorphic in the sense that there is a bijection between the sets of vertices of these two graphs which preserves the adjacency relation. Being motivated by this query, we construct an induced subgraph $G_2$ of $\Gamma'_2(\mathcal{M}(X,\mathcal{A}))$ by identifying some vertices and making a quotient like construction. Analogously, we take into our consideration the induced subgraph of $\Gamma(\mathcal{M}(X,\mathcal{A}))$ and denote it by $G$. It is realized that the graphs $G_2$ and $G$ are isomorphic. From this it is deduced by imposing certain condition on the nature of the characteristic function of measurable sets in $(X,\mathcal{A})$ that $\Gamma'_2(\mathcal{M}(X,\mathcal{A}))$ and $\Gamma(\mathcal{M}(X,\mathcal{A}))$ are isomorphic [Theorem~\ref{Iso}].\\
	In section \ref{Sec5}, we take up our study on the annihilator graph $AG(\mathcal{M}(X,\mathcal{A}))$ of the ring $\mathcal{M}(X,\mathcal{A})$. The vertices of this later graph are the same as the vertices of $\Gamma(\mathcal{M}(X,\mathcal{A}))$. Amongst several properties enjoyed by $AG(\mathcal{M}(X,\mathcal{A}))$ we would like to mention the one which characterizes the non-atomicity of the measure $\mu$. Indeed this reads: $AG(\mathcal{M}(X,\mathcal{A}))$ is hypertriangulated if and only if $\mu$ is non-atomic. This in turn leads to the fact that $AG(\mathcal{M}(X,\mathcal{A}))$ is hypertriangulated if and only if $\Gamma'_2(\mathcal{M}(X,\mathcal{A}))$ is triangulated.\\
	We make our formal study on the weakly zero-divisor graph $W\Gamma(\mathcal{M}(X,\mathcal{A}))$ of the ring $\mathcal{M}(X,\mathcal{A})$ in Section \ref{Sec6}. The set of vertices of this graph are the zero-divisors $f$ so that $Z(f)$ is an atom.\\
	In the concluding section \ref{Sec7} of this article, we illustrate various concepts and facts that we achieve in the preceding sections about three types of graphs mentioned above with the aid of two well-known specific example of measure spaces. The first case is so called Counting measure on the measurable space $(X,\mathcal{P}(X))$. The second one is the Lebesgue measure on the set $\mathbb{R}$ equipped with the $\sigma$-algebra of Lebesgue measurable sets in $\mathbb{R}$. It is realized that the zero-divisor graph and the comaximal graph of the ring of all Lebesgue measurable functions on $\mathbb{R}$ become isomorphic as graphs [Theorem~\ref{Th7.2.7}]. However it is not known to us, whether more generally for a non-atomic measure $\mu$ on $(X,\mathcal{A})$, $\Gamma'_2(\mathcal{M}(X,\mathcal{A}))$ and $\Gamma(\mathcal{M}(X,\mathcal{A}))$ are isomorphic.

	\section{\bf Preliminaries}\label{Sec2}
	\vskip 0.4 true cm

	Let $G$ be a graph with $V$ as the set of vertices of $G$. $G$ is said to be a simple graph if $G$ contains no self-loops and no parallel edges. $G$ is a connected graph if every pair of vertices is connected by a path in $G$. A stable set in $G$ is a subset of $V$ in which no two vertices are adjacent. For a cardinal number $\alpha$, $G$ is called $\alpha$-partite graph if $G$ contains $\alpha$-many disjoint stable sets whose union is $V$. An $\alpha$-partite graph is called a complete $\alpha$-partite graph if any two vertices from two different stable sets are adjacent. If $\alpha=2$, we call these graphs bipartite and complete bipartite respectively. The distance between two vertices $u,v$ in $G$ is the length of the shortest path joining $u,v$, denoted by $d(u,v)$. The eccentricity of a vertex $v$ is defined by $ecc(v)=min\{d(u,v):u\in V\}$. The diameter of $G$ is $diam(G)=max\{d(u,v):u,v\in V\}$ and the girth $gr(G)$ of $G$ is the length of the smallest cycle in $G$. $G$ is triangulated or hypertriangulated according as every vertex is a vertex of a triangle or every edge is an edge of a triangle in $G$. For two vertices $u,v\in V$, $c(u,v)$ denotes the length of the smallest cycle in $G$ containing $u,v$. Two vertices $u,v$ are said to be orthogonal in $G$, denoted by $u\perp v$, if $u,v$ are adjacent and $c(u,v)>3$. $G$ is called a complemented graph if for every $u\in V$, there exists $v\in V$ such that $u\perp v$. A complemented graph is called uniquely complemented if whenever $u\perp v$ and $u\perp w$ in $G$ for $u,v,w\in V$, then $v,w$ are adjacent to the same set of vertices in $G$. The clique number $cl(G)$ of $G$ is the maximum cardinality of complete subgraphs of $G$. A subset $V_1$ of $V$ is a dominating set in $G$ if for every $u\in V\setminus V_1$, $u,v$ are adjacent in $G$ for some $v\in V_1$ and the dominating number $dt(G)=min\{|V_1|:V_1\text{ is a dominating set in }G\}$. A subset $V_1$ of $V$ is a total dominating set in $G$ if for every $u\in V$, $u,v$ are adjacent in $G$ for some $v\in V_1$ and the total dominating number $dt_t(G)=min\{|V_1|:V_1\text{ is a total dominating set in }G\}$. For a cardinal number $\alpha$, $G$ is said to be $\alpha$-colorable if there is a map $\psi:V \to[0,\alpha]$ such that $\psi(u)\neq\psi(v)$ whenever two $u,v$ are adjacent in $G$. The chromatic number $\chi(G)$ of $G$ is $min\{\alpha:G\text{ is }\alpha\text{-colorable}\}$. For any subset $V'$ of $V$, the induced subgraph $G'$ of $G$ induced by $V'$ is the graph whose vertex set is $V'$ and two vertices in $G'$ are adjacent if they are adjacent in $G$.\\
	Let $G_1,G_2$ be two graphs having $V_1,V_2$ as set of vertices respectively. A bijection map $\psi:V_1\to V_2$ is called a graph isomorphism between two graphs $G_1,G_2$, if it preserves the adjacency relations; i.e., $a,b\in V_1$ are adjacent in $G_1$ if and only if $\psi(a),\psi(b)$ are adjacent in $G_2$. For more graph related terms we refer to \cite{Diestel}.\\
	Let $(X,\mathcal{A},\mu)$ be a measure space and $\mathcal{M}(X,\mathcal{A})$ be the corresponding ring of real-valued measurable functions on $X$. It is clear that $Z(f),X\setminus Z(f)\in\mathcal{A}$ for all $f\in\mathcal{M}(X,\mathcal{A})$, here $Z(f)=\{x\in X: f(x)=0\}$ is the zero set of $f$. For any subset $A$ of $X$, the characteristic function of $A$ is denoted by $1_A$ and is given by 
	$$1_A(x)=\begin{cases}
		1&\text{ if }x\in A\\
		0&\text{ otherwise}
	\end{cases}$$
	Therefore, $1_A\in\mathcal{M}(X,\mathcal{A})$ if and only if $A\in\mathcal{A}$. An element $A\in\mathcal{A}$ is called an atom in $X$ if $\mu(A)>0$ and $A$ can not be written as a union of two disjoint measurable sets, each with positive measure, i.e., $A$ can not be written as $B\sqcup C$, where $B,C\in\mathcal{A}$ with $\mu(B),\mu(C)>0$, here `$\sqcup$' denotes the disjoint union of two sets. Suppose that $A\in\mathcal{A}$ is an atom in $X$. Then for any $B\in\mathcal{A}$ with $\mu(B)=0$, it can easily be proved that, $A\setminus B$ and $A\cup B$ are both atoms in $X$. The measure $\mu$ is said to be an atomic measure if every measurable set with positive measure contains an atom. On the other hand, $\mu$ is said to be non-atomic if $\mathcal{A}$ does not contain any atom.\\
	Let $R$ be a commutative ring. There are several graphs, viz. the zero-divisor graph, the comaximal graph, etc., defined on $R$ to study the interaction between the properties of the ring and the properties of the respective graphs. The zero-divisor graph $\Gamma(R)$ of $R$ has its vertices as the set of non-zero zero-divisors of $R$ with two distinct vertices $x,y$ declared adjacent if and only if $x.y=0$ [\cite{Anderson1999}, \cite{Beck1988}]. The comaximal graph of $R$, defined in \cite{Sharma1995}, with vertices as elements of $R$, where two distinct vertices $x$ and $y$ are adjacent if and only if the sum of the principal ideals generated by $x,y$ is $R$. Later in \cite{Badie2016}, the vertex set of the comaximal graph of $R$ was redefined as $R\setminus [U(R)\cup J(R)]$, where $U(R)$ and $J(R)$ denote the set of all units in $R$ and the Jacobson radical of $R$ respectively. This graph is denoted by $\Gamma'_2(R)$. The annihilator graph $AG(R)$ of $R$, introduced in \cite{Badawi2014}, is a supergraph of $\Gamma(R)$ having the same set of vertices and two distinct vertices $x,y$ are adjacent if and only if $ann(x)\cup ann(y)\subsetneqq ann(x.y)$, where $ann(a)=\{r\in R:a.r=0\}$ is the annihilator ideal of $a$ in $R$. The weakly zero-divisor graph $W\Gamma(R)$ of $R$ is also a supergraph of $\Gamma(R)$ with the same set of vertices where two distinct vertices $x,y$ are adjacent if and only if there exists $a\in ann(x)\setminus\{0\}$ and $b\in ann(y)\setminus\{0\}$ such that $a.b=0$, introduced in \cite{Nikmehr2021}.\\
	As proposed in the introduction of this paper, we study the behaviour of the well-known graphs, viz. comaximal, annihilator and weakly zero-divisor graphs, redefining them via $\mu$, of the ring of measurable functions defined over a measure space $(X, \mathcal{A}, \mu)$. In this context, a few of the ring theoretic terms also demand natural modifications when `equality' is replaced by `equality almost everywhere', in the presence of a measure $\mu$ on $(X,\mathcal{A})$. We now describe this modification and redefine the terms, retaining their nomenclature in most of the cases.\\
	Two functions $f,g\in\mathcal{M}(X,\mathcal{A})$ are said to be equal almost everywhere on $X$ with respect to the measure $\mu$ if $\mu(\{x\in X:f(x)\neq g(x)\})=0$ and in this case we write $f\equiv g$ a.e. on $X$.  If $f\in\mathcal{M}(X,\mathcal{A})$ is such that $\mu(Z(f))=0$, then $f$ is said to be non-zero almost everywhere on $X$ (in short, non-zero a.e. on $X$).
	\begin{defn}\hspace{3cm} 
		\begin{enumerate}
			\item A function $f\in\mathcal{M}(X,\mathcal{A})$ which is non-zero a.e. on $X$, is called a zero-divisor in $\mathcal{M}(X,\mathcal{A})$ if there exists $g\in\mathcal{M}(X,\mathcal{A})$  such that $g$ is non-zero a.e. on $X$ and $f.g\equiv 0$ a.e. on $X$. The set of all zero-divisors in $\mathcal{M}(X,\mathcal{A})$ is denoted by $\mathscr{D}(\mathcal{M}(X,\mathcal{A}))$ (in short, $\mathscr{D}$).
			\item A function $f\in\mathcal{M}(X,\mathcal{A})$ is called a $\mu$-unit in $\mathcal{M}(X,\mathcal{A})$ if $\mu(Z(f))=0$. $U(\mathcal{M}(X,\mathcal{A}))$ stands for the set of all $\mu$-units in $\mathcal{M}(X,\mathcal{A})$.
			\item An ideal $I$ in $\mathcal{M}(X,\mathcal{A})$ is said to be almost $\mathcal{M}(X,\mathcal{A})$ if $I$ is a principal ideal generated by a $\mu$-unit in $\mathcal{M}(X,\mathcal{A})$.
		\end{enumerate}
	\end{defn}
	Since the Jacobson radical of $\mathcal{M}(X,\mathcal{A})$ is $(0)$, we redefine $J(\mathcal{M}(X,\mathcal{A}))=\{f\in\mathcal{M}(X,\mathcal{A}):f\equiv 0 \text{ a.e. on }  X \}$. For each $f\in\mathcal{M}(X,\mathcal{A})$, the annihilator ideal of $f$ is also redefined as $ann(f)=\{g\in\mathcal{M}(X,\mathcal{A}):f.g\equiv 0\text{ a.e. on }X\}$. We can also write $ann(f)$ as $\{g\in\mathcal{M}(X,\mathcal{A}):\mu(X\setminus Z(f)\cap X\setminus Z(g))=0\}$, where $f\in\mathcal{M}(X,\mathcal{A})$.

	The following theorem completely describes the zero-divisors of $\mathcal{M}(X,\mathcal{A})$ : 
	
	\begin{thm}
		$\mathscr{D}(\mathcal{M}(X,\mathcal{A}))=\{f\in\mathcal{M}(X,\mathcal{A}):\mu(Z(f))>0\text{ and }\mu(X\setminus Z(f))>0\}$.
	\end{thm}
	\begin{proof}
		Let $f\in\mathcal{M}(X,\mathcal{A})$ be such that $\mu(Z(f))>0$ and $\mu(X\setminus Z(f))>0$. Then $f$ is non-zero a.e. on $X$ and there exists $g=1_{Z(f)}$, such that $f.g=0$ on $X$. Clearly $g\in\mathcal{M}(X,\mathcal{A})$ and $\mu(X\setminus Z(g))=\mu(Z(f))>0$, i.e., $g$ is non-zero a.e. on $X$. So, $f\in\mathscr{D}$. \\
		Conversely let $f\in\mathscr{D}$. Then by definition of a zero-divisor, $\mu(X\setminus Z(f))>0$ and there exists $g\in\mathcal{M}(X,\mathcal{A})$ with $\mu(X\setminus Z(g))>0$ such that $f.g \equiv 0$ a.e., i.e., there exists $A\in\mathcal{A}$ with $\mu(A)=0$ such that $f.g=0$ on $X\setminus A$. So, $X\setminus Z(g)\cap X\setminus A\subset Z(f)$. Now, $X\setminus Z(g)=(X\setminus Z(g)\cap A)\sqcup(X\setminus Z(g)\cap X\setminus A)$ and $\mu(X\setminus Z(g)\cap A)=0$. This implies that $\mu(X\setminus Z(g)\cap X\setminus A)=\mu(X\setminus Z(g))>0$. Therefore, $\mu(Z(f))>0$. 
	\end{proof}
	
	We observe in the next theorem that the ring $\mathcal{M}(X,\mathcal{A})$ has no divisor of zero is equivalent to a purely measure theoretic phenomenon of the measure space $(X,\mathcal{A},\mu)$. 
	
	\begin{thm}\label{Th1}
		The following statements are equivalent:
		\begin{enumerate}
			\item Every measurable set in $X$ with positive measure is an atom.
			\item $X$ is an atom.
			\item $\mathscr{D}(\mathcal{M}(X,\mathcal{A}))$ is an empty set.
		\end{enumerate}
	\end{thm}
	\begin{proof}
		$(1)\implies (2)$ is trivial, as $\mu(X)>0$ and $(2)\implies (3)$ follows directly from the definition of an atom. \\
		$(3)\implies(1)$ : Let $(1)$ be false. Then some $A\in\mathcal{A}$ exists such that $\mu(A)>0$ but $A$ is not an atom. So $A$  is expressible as the disjoint union of some $A_1, A_2 \in \mathcal{A}$ with  $\mu(A_1)>0$ and $\mu(A_2)>0$. In such case, $1_{A_1},1_{A_2}\in\mathscr{D}$, proving $\mathscr{D}\neq \emptyset$.
	\end{proof}
	
	Following are a couple of examples of measure spaces where every measurable set with positive measure is an atom and therefore by Theorem~\ref{Th1}, their corresponding rings of measurable functions possess no divisor of zero.
	
	\begin{exm}
		$(1)$ Let $X$ be a  non-empty set, $\mathcal{A}=\mathscr{P}(X)$ (= the power set of $X$) and $x_0\in X$. Consider the Dirac measure at $x_0$ given by $$\delta_{x_0}(A)=\begin{cases}
			0&\text{ if }x_0\notin A\\
			1&\text{ if }x_0\in A.
		\end{cases}$$
		$(2)$ Let $X$ be an uncountable set, $\mathcal{A}=\{A\subseteqq X:\text{ either }A\text{ or }X\setminus A\text{ is countable}\}$. Consider the co-countable measure given by $$\rho(A)=\begin{cases}
			0&\text{ if }A\text{ is countable}\\
			\infty&\text{ if }X\setminus A\text{ is countable.}
		\end{cases}$$
	\end{exm}
	
	\begin{thm}\label{Lem0.4}
		For any $f,g\in\mathcal{M}(X,\mathcal{A})$, $ann(f)\subset ann(g)$ if and only if $\mu(Z(f)\setminus Z(g))=0$.
	\end{thm}
	\begin{proof}
		Let $E=Z(f)\setminus Z(g)$. If $\mu(E)=0$, then $Z(f)\cap X\setminus E\subset Z(g)$. Now let $h\in ann(f)$, i.e., $h.f=0$ on $X\setminus F$ for some $F\in\mathcal{A}$ with $\mu(F)=0$. Therefore $X\setminus Z(h)\cap X\setminus F\subset Z(f)\implies X\setminus Z(h)\cap X\setminus(E\cup F)\subset Z(g)\implies h.g\equiv 0$ on $X\setminus(E\cup F)$. Since $\mu(E\cup F)=0$, $h.g\equiv 0$ a.e. on $X$ i.e., $h\in ann(g)$. \\
		Conversely let $\mu(E)>0$. Let $h=1_E$. Then $h\in\mathcal{M}(X,\mathcal{A})$ and $X\setminus Z(h)=E=Z(f)\setminus Z(g)$. Therefore $f.h\equiv 0$ on $X$ and $X\setminus Z(h)\cap X\setminus Z(g)=E\implies\mu( X\setminus Z(h)\cap X\setminus Z(g))>0$. Thus $h\in ann(f)\setminus ann(g)$. i.e., $ann(f)\not\subset ann(g)$. 
	\end{proof}
	\begin{cor}\label{Lem6.1}
		For $f,g\in\mathscr{D}(\mathcal{M}(X,\mathcal{A}))$, $\mu(Z(f)\triangle Z(g))=0$ if and only if $ann(f)= ann(g)$.
	\end{cor}
	
	In view of Theorem~\ref{Th1}, throughout this paper, we assume that $X$ is not an atom and therefore, for every atom $A \in \mathcal{A}$, $\mu(X \setminus A)>0$. 
	
	\section{\bf The comaximal graph $\Gamma'_2(\mathcal{M}(X,\mathcal{A}))$ of $\mathcal{M}(X,\mathcal{A})$}\label{Sec3}
	We redefine the comaximal graph of $\mathcal{M}(X,\mathcal{A})$, taking into account the measure defined on the measurable space $(X, \mathcal{A})$ and  investigate in this section its graph features via the properties of the underlying measure space $(X, \mathcal{A}, \mu)$  and the ring $\mathcal{M}(X, \mathcal{A})$. 
	\begin{defn} 
		The comaximal graph $\Gamma'_2(\mathcal{M}(X,\mathcal{A}))$ (in short, $\Gamma'_2$) of $\mathcal{M}(X,\mathcal{A})$ is a graph on the vertex set $\mathcal{M}(X,\mathcal{A})\setminus[U(\mathcal{M}(X,\mathcal{A}))\cup J(\mathcal{M}(X,\mathcal{A}))]$ and the adjacency relation for two vertices $ f,g $ is given by :   $ f,g $ are adjacent if and only if the sum of the principal ideals generated by $f,g$ is almost $\mathcal{M}(X,\mathcal{A})$. In notation, $<f> + <g> = <u>$, where $u$ is a $\mu$-unit in $\mathcal{M}(X,\mathcal{A})$.
	\end{defn}
	Theorem~\ref{Th1} suggests that the vertex set of the comaximal graph $\Gamma'_2(\mathcal{M}(X,\mathcal{A}))$ is  $\mathscr{D}(\mathcal{M}(X,\mathcal{A}))$ and hence it is a non-empty graph under the assumption that $X$ is not an atom. In \cite{Hejazipour2022}, the author defined the zero-divisor graph $\Gamma(\mathcal{M}(X,\mathcal{A}))$ (in short, $\Gamma$) of $\mathcal{M}(X,\mathcal{A})$ whose vertex set is $\mathscr{D}(\mathcal{M}(X,\mathcal{A}))$ and the adjacency relation is given by: $f$ is adjacent to $g$ if and only if $\mu(X\setminus Z(f)\cap X\setminus Z(g))=0$. So, the sets of vertices of $\Gamma'_2(\mathcal{M}(X,\mathcal{A}))$ and $\Gamma(\mathcal{M}(X,\mathcal{A}))$ coincide.\\
	The adjacency relation in $\Gamma'_2(\mathcal{M}(X,\mathcal{A}))$ can be interpreted through the almost disjointness of the corresponding zero sets of the vertices in $\Gamma'_2(\mathcal{M}(X,\mathcal{A}))$, as seen in the next result.
	\begin{thm}\label{thm}
		Let $f,g\in\mathscr{D}(\mathcal{M}(X,\mathcal{A}))$. Then $f,g$ are adjacent in $\Gamma'_2(\mathcal{M}(X,\mathcal{A}))$ if and only if $\mu(Z(f)\cap Z(g))=0$.
	\end{thm}
	\begin{proof}
		Let $f,g$ be adjacent in $\Gamma'_2$. Then there exists a $\mu$-unit $u$ in $\mathcal{M}(X,\mathcal{A})$ such that $<f>+<g>=<u>$. i.e., $f.f_1+g.g_1=u$ for some $f_1,g_1\in\mathcal{M}(X,\mathcal{A})$. Therefore, $Z(f)\cap Z(g)\subset Z(u)$. Since $u$ is a $\mu$-unit $\mu(Z(u))=0\implies\mu(Z(f)\cap Z(g))=0$. \\
		Conversely, let $\mu(Z(f)\cap Z(g))=0$. We claim that $<f>+<g>=<u>$, where $u=f^2+g^2$, a $\mu$-unit in $\mathcal{M}(X,\mathcal{A})$. Clearly $u\in<f>+<g>$. Also $f\in<u>$, because $f=u.f_1$, where $$f_1(x)=\begin{cases}
			\frac{f(x)}{f^2(x)+g^2(x)}&\text{ if }x\notin Z(f)\cap Z(g)\\
			0&\text{ if }x\in Z(f)\cap Z(g)
		\end{cases}$$ and similarly, $g\in<u>$. Hence, $<f>+<g>=<u>$ . i.e., $f,g$ are adjacent in $\Gamma'_2$.
	\end{proof}
	The following theorem provides a necessary and sufficient condition under which a pair of vertices of $\Gamma'_2(\mathcal{M}(X,\mathcal{A}))$ admits of a third vertex adjacent to both of them.  
	\begin{thm}\label{Lem2.1}
		Let $f,g\in\mathscr{D}(\mathcal{M}(X,\mathcal{A}))$. Then there is a vertex in $\mathscr{D}(\mathcal{M}(X,\mathcal{A}))$ adjacent to both $f,g$ in $\Gamma'_2(\mathcal{M}(X,\mathcal{A}))$ if and only if $\mu(X\setminus Z(f)\cap X\setminus Z(g))>0$.
	\end{thm}
	\begin{proof}
		Let $\mu(X\setminus Z(f)\cap X\setminus Z(g))>0$. Consider $h=1_{Z(f)\cup Z(g)}\in\mathcal{M}(X,\mathcal{A})$. Clearly $Z(h)=X\setminus Z(f)\cap X\setminus Z(g)\implies h\in\mathscr{D}$. Also $Z(f)\cap Z(h)=\emptyset=Z(g)\cap Z(h)$ i.e., $h$ is adjacent to both $f,g$ in $\Gamma'_2$. Conversely, let $h\in\mathscr{D}$ be adjacent to both $f,g$ in $\Gamma'_2$. So $\mu(Z(f)\cap Z(h))=0=\mu(Z(g)\cap Z(h))$ and hence, $\mu(Z(h)\cap[Z(f)\cup Z(g)])=0$. Now $Z(h)=(Z(h)\cap[Z(f)\cup Z(g)])\sqcup(Z(h)\cap X\setminus[Z(f)\cup Z(g)])\implies \mu(Z(h)\cap X\setminus[Z(f)\cup Z(g)])=\mu(Z(h))>0$, as $h\in\mathscr{D}$. Consequently $\mu(X\setminus[Z(f)\cup Z(g)])>0$; i.e., $\mu(X\setminus Z(f)\cap X\setminus Z(g))>0$.
	\end{proof}
	\begin{cor}
		Every edge $f-g$ in $\Gamma'_2(\mathcal{M}(X,\mathcal{A}))$ is either an edge of a triangle or an edge of a square according as $\mu(X\setminus Z(f)\cap X\setminus Z(g))>0$ or $\mu(X\setminus Z(f)\cap X\setminus Z(g))=0$.
	\end{cor}
	\begin{proof}
		If $\mu(X\setminus Z(f)\cap X\setminus Z(g))>0$, then by Theorem \ref{Lem2.1}, there exists a vertex adjacent to both $f,g$. If $\mu(X\setminus Z(f)\cap X\setminus Z(g))=0$, then $f-g-2f-2g-f$ forms a square in $\Gamma'_2$.
	\end{proof}
	\begin{cor}
		$gr(\Gamma'_2(\mathcal{M}(X,\mathcal{A})))\leq 4$.
	\end{cor}
	The distance between two vertices in $\Gamma'_2(\mathcal{M}(X,\mathcal{A}))$ is given below.
	\begin{thm}\label{Th2.3}
		Let $f,g\in\mathscr{D}(\mathcal{M}(X,\mathcal{A}))$. Then 
		$$d(f,g)=\begin{cases}
			1&\text{ if }\mu(Z(f)\cap Z(g))=0\\
			2&\text{ if }\mu(Z(f)\cap Z(g))>0\text{ and }\mu(X\setminus Z(f)\cap X\setminus Z(g))>0\\
			3&\text{ if }\mu(Z(f)\cap Z(g))>0\text{ and }\mu(X\setminus Z(f)\cap X\setminus Z(g))=0
		\end{cases}$$
	\end{thm}
	\begin{proof}
		The conditions for $d(f,g)=1$ or $2$ are straightforward consequences of the adjacency relation and Theorem \ref{Lem2.1}. From these two results, $d(f,g)=3\implies \mu(Z(f)\cap Z(g))>0$ and $\mu(X\setminus Z(f)\cap X\setminus Z(g))=0$. Conversely let $\mu(Z(f)\cap Z(g))>0$ and $\mu(X\setminus Z(f)\cap X\setminus Z(g))=0$. Then certainly, $d(f,g)>2$. Let $h_1=1_{Z(f)}$ and $h_2=1_{Z(g)}$. Then $h_1,h_2\in\mathscr{D}$ and $Z(h_1)=X\setminus Z(f), Z(h_2)=X\setminus Z(g)\implies Z(f)\cap Z(h_1)=\emptyset=Z(h_2)\cap Z(g)$ and also $\mu(Z(h_1)\cap Z(h_2))=0$. Therefore $f-h_1-h_2-g$ is a path in $\Gamma'_2$. Consequently, $d(f,g)=3$.
	\end{proof}
	\begin{cor}
		$\Gamma'_2(\mathcal{M}(X,\mathcal{A}))$ is a connected graph.
	\end{cor}
	\begin{lem}\label{Lem3.13}
		Let $f,g\in\mathscr{D}(\mathcal{M}(X,\mathcal{A}))$.
		\begin{enumerate}
			\item If $\mu(Z(f)\triangle Z(g))=0$, then $f,g$ are not adjacent in $\Gamma'_2(\mathcal{M}(X,\mathcal{A}))$.\label{Lem3.13.1}
			\item If $\mu(Z(f)\setminus A)=0$ and $\mu(Z(g)\cap A)=0$ for some $A\in\mathcal{A}$, then $f,g$ are adjacent in $\Gamma'_2(\mathcal{M}(X,\mathcal{A}))$.\label{Lem3.13.2}
			\item If $f,g$ are adjacent in $\Gamma'_2(\mathcal{M}(X,\mathcal{A}))$, then for any $f_1\in\mathscr{D}(\mathcal{M}(X,\mathcal{A}))$ with $\mu(Z(f_1)\triangle Z(f))=0$ and for any $g_1\in\mathscr{D}(\mathcal{M}(X,\mathcal{A}))$ with $\mu(Z(g_1)\triangle Z(g))=0$, $f_1,g_1$ are adjacent in $\Gamma'_2(\mathcal{M}(X,\mathcal{A}))$.\label{Lem3.13.3}
		\end{enumerate}
	\end{lem}
	\begin{proof}
		\hspace*{3cm}
		\begin{enumerate}
			\item Let $\mu(Z(f)\triangle Z(g))=0$. If $f,g$ are adjacent in $\Gamma'_2$, then $\mu(Z(f)\cap Z(g))=0$. Now $Z(f)\subset (Z(f)\cap Z(g))\cup(Z(f)\triangle Z(g))\implies\mu(Z(f))=0$, a contradiction to $f\in\mathscr{D}$. Therefore, $f,g$ are not adjacent in $\Gamma'_2$.
			\item Let $\mu(Z(f)\setminus A)=0$ and $\mu(Z(g)\cap A)=0$ for some $A\in\mathcal{A}$. Now $Z(f)\cap Z(g)\subset [A\cup (Z(f)\setminus A)]\cap [(X\setminus A)\cup (Z(g)\cap A)]\implies\mu(Z(f)\cap Z(g))=0$. Therefore, $f,g$ are adjacent in $\Gamma'_2$.
			\item Let $f,g$ be adjacent in $\Gamma'_2$ and $\mu(Z(f_1)\triangle Z(f))=0=\mu(Z(g_1)\triangle Z(g))$ for some $f_1,g_1\in\mathscr{D}$. Since $f,g$ are adjacent, $\mu(Z(f)\cap Z(g))=0$. Now $Z(f_1)\cap Z(g_1)\subset [Z(f)\cup(Z(f_1)\triangle Z(f))]\cap[Z(g)\cup(Z(g_1)\triangle Z(g))]$. Therefore, $f_1,g_1$ are adjacent in $\Gamma'_2$.
		\end{enumerate}
	\end{proof}
	\begin{lem}\label{Lem3.14}
		A partition of $X$ by two atoms $A,B$ yields a partition of $\mathscr{D}(\mathcal{M}(X,\mathcal{A}))$ by two sets $\{f\in\mathscr{D}(\mathcal{M}(X,\mathcal{A})):\mu(Z(f)\triangle A)=0\}$ and $\{f\in\mathscr{D}(\mathcal{M}(X,\mathcal{A})):\mu(Z(f)\triangle B)=0\}$.
	\end{lem}
	\begin{proof}
		To prove this result, it is enough to show that for each $f\in \mathscr{D}$, either $\mu(Z(f)\triangle A)=0$ or $\mu(Z(f)\triangle B)=0$.
		Let $f\in\mathscr{D}$. Since $A$ is an atom, either $\mu(A\cap Z(f))=0$ or $\mu(A\cap X\setminus Z(f))=0$. At first, let $\mu(A\cap Z(f))=0$. By hypothesis, $A = X \setminus B$ and so, $\mu(X\setminus B\cap Z(f))=0$. If $\mu(B\cap Z(f))=0$, then $\mu(Z(f)) =0$ which contradicts $f\in\mathscr{D}$. Hence, $\mu(B\cap Z(f))>0$  which implies $\mu(B\cap X\setminus Z(f))=0$, as $B$ is also an atom. Consequently, $\mu(B\triangle Z(f))=0$. Analogously, the hypothesis $\mu(A\cap X\setminus Z(f))=0$ would lead to the conclusion $\mu(Z(f)\triangle A)=0$.
	\end{proof}
	
	\begin{thm}\label{Th3.13}
		$\Gamma'_2(\mathcal{M}(X,\mathcal{A}))$ is a complete bipartite graph if and only if $X$ is partitioned into two atoms.
	\end{thm}
	\begin{proof}
		If $\Gamma'_2$ is a complete bipartite graph, then it has a bipartion, say $V_1,V_2$. Since both $V_1,V_2$ are non-empty, we choose and fix $f\in V_1$ and $g\in V_2$. $\Gamma'_2$ being complete bipartite, no $h\in\mathscr{D}$ is adjacent to both $f,g$ and hence by Lemma \ref{Lem2.1}, $\mu(X\setminus Z(f)\cap X\setminus Z(g))=0$. Also $f,g$ are adjacent in $\Gamma'_2\implies\mu(Z(f)\cap Z(g))=0$. We claim that $Z(f)$ is an atom. To prove this, let $A\in\mathcal{A}$ with $\mu(A)>0$. If $\mu(X\setminus A)=0$, then $\mu(X\setminus A\cap Z(f))=0$. Now let $\mu(X\setminus A)>0$. Then $1_A,1_{X\setminus A}\in\mathscr{D}$ and $1_A,1_{X\setminus A}$ are adjacent in $\Gamma'_2$. If $1_A\in V_1$, then $1_{X\setminus A}\in V_2\implies 1_{X\setminus A},f$ are adjacent i.e., $\mu(A\cap Z(f))=0$. If $1_A\in V_2$, then $1_A,f$ are adjacent and so $\mu(X\setminus A\cap Z(f))=0$. In any case either $\mu(A\cap Z(f))=0$ or $\mu(X\setminus A\cap Z(f))=0$. Consequently, $Z(f)$ is an atom. Similarly we can show that $Z(g)$ is an atom. Let $A=Z(f)\setminus Z(g)=Z(f)\setminus[Z(f)\cap Z(g)]$ and $B= Z(g)\cup X\setminus Z(f)=Z(g)\cup[X\setminus Z(f)\cap X\setminus Z(g)]$. Then $X=A\sqcup B$ and $A,B$ are atoms. Conversely let $A,B$ be two atoms in $X$ such that $X=A\sqcup B$. Let $V_1=\{f\in\mathscr{D}:\mu(Z(f)\triangle A)=0\}$ and $V_2=\{f\in\mathscr{D}:\mu(Z(f)\triangle B)=0\}$. By Lemma \ref{Lem3.14}, $\mathscr{D}=V_1\sqcup V_2$. Again by Lemma \ref{Lem3.13}(\ref{Lem3.13.1}), $V_1,V_2$ are stable sets in $\Gamma'_2$. Let $f\in V_1$ and $g\in V_2$. Then $\mu(Z(f)\triangle A)=0=\mu(Z(g)\triangle B)\implies\mu(Z(f)\setminus A)=0=\mu(Z(g)\setminus B)=\mu(Z(g)\cap A)$, as $B=X\setminus A$. By Lemma \ref{Lem3.13}(\ref{Lem3.13.2}), $f,g$ are adjacent in $\Gamma'_2$. Consequently, $\Gamma'_2$ is a complete bipartite graph.
	\end{proof}
	\begin{cor}
		$X$ is partitioned into two atoms if and only if for each $f\in\mathscr{D}(\mathcal{M}(X,\mathcal{A}))$, both $Z(f)$ and $X\setminus Z(f)$ are atoms.
	\end{cor}
	\begin{proof}
		Let $X=A\sqcup B$, where $A,B$ are two atoms. Then by Theorem \ref{Th3.13}, $\mu(Z(f)\triangle A)=0$ or $\mu(Z(f)\triangle B)=0$. Without loss of generality let $\mu(Z(f)\triangle A)=0\implies\mu(Z(f)\setminus A)=0$. Now $Z(f)=A\setminus(A\triangle Z(f))\cup Z(f)\setminus A\implies Z(f)$ is an atom. Let $g=1_{Z(f)}$. Then $g\in\mathscr{D}$ and $Z(g)=X\setminus Z(f)$. Similarly we can show that $Z(g)$ i.e., $X\setminus Z(f)$ is an atom. Conversely let $X$ can not be partitioned into two atoms and $f\in\mathscr{D}$. Since $X=Z(f)\sqcup X\setminus Z(f)$, either $Z(f)$ or $X\setminus Z(f)$ is not an atom. 
	\end{proof}
	Replacing $Z(f)$ by $X\setminus Z(f)$ in the proof of Theorem~\ref{Th3.13} and making some suitable modification we can prove that $\Gamma(\mathcal{M}(X,\mathcal{A}))$ is a complete bipartite graph if and only if $X$ is partitioned into two atoms. This observation is recorded in the following theorem. 
	\begin{thm}\label{CompBip}
		The following statements are equivalent:
		\begin{enumerate}
			\item $\Gamma(\mathcal{M}(X,\mathcal{A}))$ is a complete bipartite graph.
			\item $\Gamma'_2(\mathcal{M}(X,\mathcal{A}))$ is a complete bipartite graph.
			\item $X$ is partitioned into two atoms.
		\end{enumerate}
	\end{thm}
	In the next few theorems we find the eccentricity, diameter and the lengths of possible cycles in $\Gamma'_2(\mathcal{M}(X,\mathcal{A}))$ and determine how far the measure $\mu$ is responsible to make $\Gamma'_2(\mathcal{M}(X,\mathcal{A}))$ triangulated, hypertriangulated and complemented.
	\begin{thm}\label{Th2.6}
		Let $f\in\mathscr{D}(\mathcal{M}(X,\mathcal{A}))$. Then 
		$$ecc(f)=\begin{cases}
			2&\text{ if }Z(f)\text{ is an atom}\\
			3&\text{ otherwise}
		\end{cases}$$
	\end{thm}
	\begin{proof}
		Let $Z(f)$ be an atom and $g\in\mathscr{D}$. Then either $\mu(Z(f)\cap Z(g))=0$ or $\mu(Z(f)\cap X\setminus Z(g))=0$. If $\mu(Z(f)\cap Z(g))=0$, then $f,g$ are adjacent. If $\mu(Z(f)\cap X\setminus Z(g))=0$, then $\mu(X\setminus Z(f)\cap X\setminus Z(g))>0$, for otherwise $\mu(X\setminus Z(g))=0$. By Theorem \ref{Th2.3}, $d(f,g)=2$. Thus $ecc(f)=2$. Conversely, let $Z(f)$ be not an atom i.e., there exist $A,B\in\mathcal{A}$ with $\mu(A),\mu(B)>0$ such that $Z(f)=A\sqcup B$. Clearly, $1_A\in\mathscr{D}$, $Z(f)\cap Z(1_A)=B$ and $X\setminus Z(f)\cap X\setminus Z(1_A)=\emptyset$ i.e., $\mu(Z(f)\cap Z(1_A))>0$ and $\mu(X\setminus Z(f)\cap X\setminus Z(1_A)=0$. By Theorem \ref{Th2.3}, $d(f,1_A)=3\implies ecc(f)=3$.
	\end{proof}
	\begin{cor}
		The diameter of $\Gamma'_2(\mathcal{M}(X,\mathcal{A}))$ is $$diam(\Gamma'_2(\mathcal{M}(X,\mathcal{A})))=\begin{cases}
			2&\text{ if }X\text{ is partitioned into two atoms}\\
			3&\text{ otherwise}
		\end{cases}$$
	\end{cor}
	\begin{proof}
		If $X$ is partitioned into two atoms, then by Theorem \ref{Th3.13}, $\Gamma'_2$ is a complete bipartite graph and hence $diam(\Gamma'_2)=2$. If $X$ can not be partitioned into two atoms, then for any $f\in\mathscr{D}$, either $Z(f)$ or $X\setminus Z(f)$ is not an atom. By Theorem \ref{Th2.6}, $ecc(f)=3$, if $Z(f)$ is not an atom and $ecc(1_{Z(f)})=3$, if $X\setminus Z(f)$ is not an atom. In any case, the diameter of $\Gamma'_2$ is $3$.
	\end{proof}
	
	\begin{thm}\label{Th2.9}
		$f\in\mathscr{D}(\mathcal{M}(X,\mathcal{A}))$ is a vertex of a triangle in $\Gamma'_2(\mathcal{M}(X,\mathcal{A}))$ if and only if $X\setminus Z(f)$ is not an atom.
	\end{thm}
	\begin{proof}
		If $X\setminus Z(f)$ is not an atom, then there exist $A,B\in\mathcal{A}$ with $\mu(A),\mu(B)>0$ and $X\setminus Z(f)=A\sqcup B$. Let $g=1_{X\setminus A}$ and $h=1_{X\setminus B}$. Then $g,h\in\mathscr{D}$ and $Z(g)=A, Z(h)=B\implies Z(f)\cap Z(g)=\emptyset=Z(g)\cap Z(h)=Z(h)\cap Z(f)$. Thus $f-g-h-f$ is a triangle in $\Gamma'_2$. Conversely if $f$ is a vertex of a triangle in $\Gamma'_2$, then there exists $g,h\in\mathscr{D}$ such that $f-g-h-f$ is a triangle in $\Gamma'_2$. Since $f,g$ are adjacent, $\mu(Z(f)\cap Z(g))=0\implies\mu(X\setminus Z(f)\cap Z(g))>0$, for otherwise $\mu(Z(g))=0$. Similarly the adjacency of $f,h$ implies $\mu(X\setminus Z(f)\cap Z(h))>0$. Let $A=X\setminus Z(f)\cap Z(g)$ and $B=X\setminus Z(f)\cap X\setminus Z(g)$. Then $\mu(A)>0$ and $X\setminus Z(f)=A\sqcup B$. Since $g,h$ are adjacent, $\mu(Z(g)\cap Z(h))=0\implies\mu(A\cap Z(h))=0$. Now $X\setminus Z(f)\cap Z(h)=(A\cap Z(h))\sqcup(B\cap Z(h))\implies \mu(B\cap Z(h))=\mu(X\setminus Z(f)\cap Z(h))>0\implies\mu(B)>0$. Thus $X\setminus Z(f)=A\sqcup B$ and $\mu(A),\mu(B)>0$, which prove that $X\setminus Z(f)$ is not an atom.
	\end{proof}
	\begin{cor}
		The girth of $\Gamma'_2(\mathcal{M}(X,\mathcal{A}))$ is $$gr(\Gamma'_2(\mathcal{M}(X,\mathcal{A})))=\begin{cases}
			4&\text{ if }X\text{ is partitioned into two atoms}\\
			3&\text{ otherwise}
		\end{cases}$$.
	\end{cor}
	\begin{proof}
		If $X$ is partitioned into two atoms, then by Theorem \ref{Th3.13}, $\Gamma'_2$ is a complete bipartite graph and hence $gr(\Gamma'_2)=4$. If $X$ can not be partitioned into two atoms, then for each $f\in\mathscr{D}$, either $Z(f)$ or $X\setminus Z(f)$ is not an atom. Consequently, by Theorem \ref{Th2.9}, $1_{Z(f)}$ or $f$ is a vertex of a triangle in $\Gamma'_2$. So, $gr(\Gamma'_2)=3$. 
	\end{proof}
	\begin{thm}\label{Th3.22}
		$\Gamma'_2(\mathcal{M}(X,\mathcal{A}))$ is triangulated if and only if $\mu$ is non-atomic.
	\end{thm}
	\begin{proof}
		If $\mu$ is non-atomic, then $X\setminus Z(f)$ is not an atom for all $f\in\mathscr{D}$. By Theorem \ref{Th2.9}, every $f\in\mathscr{D}$ is a vertex of a triangle. In other words, $\Gamma'_2$ is triangulated. Conversely let $X$ contain an atom, say $A$. Then $\mu(X\setminus A)>0$, for otherwise, $X$ will be an atom. Thus $1_A\in\mathscr{D}$ and $X\setminus Z(1_A)=A$ is an atom. Hence by Theorem \ref{Th2.9}, $1_A$ is not a vertex of a triangle; i.e., $\Gamma'_2$ is not triangulated.
	\end{proof}
	\begin{thm}
		Given $f\in\mathscr{D}(\mathcal{M}(X,\mathcal{A}))$, there exists an edge containing $f$ which is not an edge of any triangle in $\Gamma'_2(\mathcal{M}(X,\mathcal{A}))$.
	\end{thm}
	\begin{proof}
		For each $f\in\mathscr{D}$, consider $g=1_{Z(f)}\in\mathscr{D}$. Then $f,g$ are adjacent and $\mu(X\setminus Z(f)\cap X\setminus Z(g))=0$. By Theorem \ref{Lem2.1}, there does not exist any vertex in $\mathscr{D}$ adjacent to both $f$ and $g$. Therefore, the edge $f-g$ is not an edge of any triangle in $\Gamma'_2$.
	\end{proof}
	\begin{cor}\label{NotHyp}
		$\Gamma'_2(\mathcal{M}(X,\mathcal{A}))$ is not hypertriangulated.
	\end{cor}
	If $X$ is partitioned into two atoms, then by Theorem \ref{Th3.13}, $\Gamma'_2$ is a complete bipartite graph and hence $c(f,g)=4$ for all $f,g\in\mathscr{D}(\mathcal{M}(X,\mathcal{A}))$. We now calculate $c(f,g)$ for $f,g\in\mathscr{D}(\mathcal{M}(X,\mathcal{A}))$, when $X$ is not partitioned into two atoms.
	\begin{thm}
		Suppose $X$ can not be partitioned into two atoms and $f,g\in\mathscr{D}(\mathcal{M}(X,\mathcal{A}))$. Then 
		$$c(f,g)=\begin{cases}
			3&\text{ if }\mu(Z(f)\cap Z(g))=0\text{ and }\mu(X\setminus Z(f)\cap X\setminus Z(g))>0\\
			4&\text{ if }\mu(Z(f)\cap Z(g))=0\text{ and }\mu(X\setminus Z(f)\cap X\setminus Z(g))=0\\
			&\text{ or if }\mu(Z(f)\cap Z(g))>0\text{ and }\mu(X\setminus Z(f)\cap X\setminus Z(g))>0\\
			6&\text{ if }\mu(Z(f)\cap Z(g))>0\text{ and }\mu(X\setminus Z(f)\cap X\setminus Z(g))=0
		\end{cases}$$
	\end{thm}
	\begin{proof}
		It is clear that $c(f,g)=3$ if and only if $f-g$ is an edge of a triangle. So, $c(f,g)=3$  if and only if $f, g$ are adjacent in $\Gamma'_2$ and there exists a vertex in $\Gamma'_2$ adjacent to both $f$ and $g$;  i.e., by Theorem \ref{thm} and Theorem \ref{Lem2.1}, $\mu(Z(f)\cap Z(g))=0$ and $\mu(X\setminus Z(f)\cap X\setminus Z(g))>0$.\\
		Let $c(f,g)=4$. If $f,g$ are adjacent  in $\Gamma'_2$, then there exists no vertex adjacent to both $f$ and $g$, for otherwise $c(f,g)=3$. Therefore, $\mu(Z(f)\cap Z(g))=0$ and $\mu(X\setminus Z(f)\cap X\setminus Z(g))=0$. If $f,g$ are not adjacent, then there exists vertex adjacent to both $f,g$ in $\Gamma'_2$, because $c(f,g)=4$. Therefore, $\mu(Z(f)\cap Z(g))>0$ and $\mu(X\setminus Z(f)\cap X\setminus Z(g))>0$. Conversely let $\mu(Z(f)\cap Z(g))=0$ and $\mu(X\setminus Z(f)\cap X\setminus Z(g))=0$. Then $f,g$ are adjacent and there does not exist any vertex adjacent to both $f,g\implies c(f,g)>3$. Clearly $f-g-2.f-2.g-f$ is a $4$-cycle in $\Gamma'_2$ which implies $c(f,g)=4$. Again let $\mu(Z(f)\cap Z(g))>0$ and $\mu(X\setminus Z(f)\cap X\setminus Z(g))>0$. Then $f,g$ are non-adjacent and there exists $h\in\mathscr{D}$ adjacent to both $f,g$ in $\Gamma'_2$. Clearly, $c(f,g)>3$ and $f-h-g-2.h-f$ is a $4$-cycle in $\Gamma'_2\implies c(f,g)=4$.\\
		If $c(f,g)=6$, then it follows from the previous two cases that $\mu(Z(f)\cap Z(g))>0$ and $\mu(X\setminus Z(f)\cap X\setminus Z(g))=0$. Conversely let $\mu(Z(f)\cap Z(g))>0$ and $\mu(X\setminus Z(f)\cap X\setminus Z(g))=0$. So, $f,g$ are not adjacent and there does not exist any vertex adjacent to both $f,g$. Hence there does not exist any $5$-cycle in $\Gamma'_2$ containing $f,g$ as vertices. In other words, $c(f,g)>5$. By Theorem \ref{Th2.3}, $d(f,g)=3$;  i.e., there exist $h_1,h_2\in\mathscr{D}$ such that $f-h_1-h_2-g$ is a path in $\Gamma'_2$. Consequently, $f-h_1-h_2-g-2.h_2-2.h_1-f$ is a $6$-cycle in $\Gamma'_2\implies c(f,g)=6$. 
	\end{proof}
	The visual representation of the above Theorem is exhibited as follows:\\
	\begin{center}
		\begin{tikzpicture}
			\draw (0,0)--(2,2)--(4,0)--(0,0) (7,0)--(7,2)--(11,2)--(11,0)--(7,0);
			\filldraw[black] (0,0) circle (2pt) node[anchor=east]{$f$};
			\filldraw[black] (4,0) circle (2pt) node[anchor=west]{$g$};
			\filldraw[black] (7,0) circle (2pt) node[anchor=east]{$f$};
			\filldraw[black] (11,0) circle (2pt) node[anchor=west]{$g$};
			\node (start) at (2,-1) {$\mu(X\setminus Z(f)\cap X\setminus Z(g))>0$};
			\node (start) at (2,-1.5) {$\text{ and }f,g\text{ are adjacent}$};
			\node (start) at (9,-1) {$\mu(X\setminus Z(f)\cap X\setminus Z(g))=0$};
			\node (start) at (9,-1.5) {$\text{ and }f,g\text{ are adjacent}$};
			\draw (0,-4)--(2,-3)--(4,-4)--(2,-5)--(0,-4) (7,-4)--(8,-3)--(10,-3)--(11,-4)--(10,-5)--(8,-5)--(7,-4) (8,-3)--(10,-5) (10,-3)--(8,-5);
			\filldraw[black] (0,-4) circle (2pt) node[anchor=east]{$f$};
			\filldraw[black] (4,-4) circle (2pt) node[anchor=west]{$g$};
			\filldraw[black] (7,-4) circle (2pt) node[anchor=east]{$f$};
			\filldraw[black] (11,-4) circle (2pt) node[anchor=west]{$g$};
			\node (start) at (2,-6) {$\mu(X\setminus Z(f)\cap X\setminus Z(g))>0$};
			\node (start) at (2,-6.5) {$\text{ and }f,g\text{ are non-adjacent}$};
			\node (start) at (9,-6) {$\mu(X\setminus Z(f)\cap X\setminus Z(g))=0$};
			\node (start) at (9,-6.5) {$\text{ and }f,g\text{ are non-adjacent}$};
		\end{tikzpicture}
	\end{center}
	From the definition of orthogonality and from Theorem \ref{Lem2.1}, it naturally follows that for $f,g\in\mathscr{D}(\mathcal{M}(X,\mathcal{A}))$, $f\perp g$ in $\Gamma'_2(\mathcal{M}(X,\mathcal{A}))$ if and only if $\mu(Z(f)\cap Z(g))=0=\mu(X\setminus Z(f)\cap X\setminus Z(g))$. As a result, for each $f\in \mathscr{D}(\mathcal{M}(X,\mathcal{A}))$, $f\perp 1_{Z(f)}$, where $1_{Z(f)} \in \mathscr{D}(\mathcal{M}(X,\mathcal{A}))$. This observation leads to the following fact.
	\begin{thm}\label{ComaxComp}
		$\Gamma'_2(\mathcal{M}(X,\mathcal{A}))$ is a complemented graph.
	\end{thm}
	\begin{lem}\label{Lem2.17}
		Let $f,g\in\mathscr{D}(\mathcal{M}(X,\mathcal{A}))$. Then $f,g$ are adjacent to the same set of vertices in $\Gamma'_2(\mathcal{M}(X,\mathcal{A}))$ if and only if $\mu(Z(f)\triangle Z(g))=0$.
	\end{lem}
	\begin{proof}
		Let $\mu(Z(f)\triangle Z(g))=0$. Then $\mu(Z(f)\cap X\setminus Z(g))=0=\mu(X\setminus Z(f)\cap Z(g))$. Let $h\in\mathscr{D}$ be adjacent to $f$, i.e., $\mu(Z(f)\cap Z(h))=0$. Now $Z(g)=(Z(g)\cap Z(f))\sqcup(Z(g)\cap X\setminus Z(f))\implies Z(g)\cap Z(h)\subset [Z(g)\cap X\setminus Z(f)]\cup[Z(f)\cap Z(h)]\implies\mu(Z(g)\cap Z(h))=0$; i.e., $g,h$ are adjacent. Similarly, if $h\in\mathscr{D}$ is adjacent to $g$, then it can also be adjacent to $f$. Therefore, $f,g$ are adjacent to the same set of vertices in $\Gamma'_2$. Conversely let $\mu(Z(f)\triangle Z(g))>0$. Then either $\mu(Z(f)\cap X\setminus Z(g))>0$ or $\mu(X\setminus Z(f)\cap Z(g))>0$. Without loss of generality let $\mu(Z(f)\cap X\setminus Z(g))>0$. Choose $h=1_{X\setminus Z(f)\cup Z(g)} \in\mathscr{D}$. Then $Z(h)=Z(f)\cap X\setminus Z(g)\implies Z(h)\cap Z(g)=\emptyset$ and $Z(h)\subset Z(f)\implies \mu(Z(h)\cap Z(g))=0$ and $\mu(Z(h)\cap Z(f))=\mu(Z(h))>0$. Hence $h$ is adjacent to $g$ but $h$ is not adjacent to $f$.
	\end{proof}
	\begin{thm}
		$\Gamma'_2(\mathcal{M}(X,\mathcal{A}))$ is a uniquely complemented graph.
	\end{thm}
	\begin{proof}
		Let $f,g,h\in\mathscr{D}$ be such that $f\perp g$ and $f\perp h$ in $\Gamma'_2$. We claim that $\mu(Z(g)\triangle Z(h))=0$. If possible let $\mu(Z(g)\cap X\setminus Z(h))>0$, then $k=1_{X\setminus Z(g)\cup Z(h)}\in\mathscr{D}$. Now $Z(k)=Z(g)\cap X\setminus Z(h)$ and so $Z(k)\cap Z(h)=\emptyset\implies\mu(Z(k)\cap Z(h))=0$ and also $Z(k)\cap Z(f)\subset Z(g)\cap Z(f)\implies\mu(Z(k)\cap Z(f))=0$, as $f\perp g$. Therefore $k$ is adjacent to both $f,h$, which contradicts that $f\perp h$. Therefore, $\mu(Z(g)\cap X\setminus Z(h))=0$ and similarly $\mu(X\setminus Z(g)\cap Z(h))=0$. Hence $\mu(Z(g)\triangle Z(h))=0$. By Lemma \ref{Lem2.17}, $g,h$ are adjacent to the same set of vertices in $\Gamma'_2$. Consequently, $\Gamma'_2$ is uniquely complemented.
	\end{proof}
	
	We now introduce special subgraphs of $\Gamma(\mathcal{M}(X,\mathcal{A}))$ and $\Gamma'_2(\mathcal{M}(X,\mathcal{A}))$ which enable us to find conditions under which $\Gamma(\mathcal{M}(X,\mathcal{A}))$ and $\Gamma'_2(\mathcal{M}(X,\mathcal{A}))$ become isomorphic. On $\mathscr{D}(\mathcal{M}(X,\mathcal{A}))$, we define a relation `` $f\sim g$ if and only if $\mu(Z(f)\triangle Z(g))=0$''. It is easy to see that $\sim$ is an equivalence relation on $\mathscr{D}(\mathcal{M}(X,\mathcal{A}))$ and hence it partitions $\mathscr{D}(\mathcal{M}(X,\mathcal{A}))$ into disjoint equivalence classes, denoted by$[f]$, for each $f\in\mathscr{D}(\mathcal{M}(X,\mathcal{A}))$. Certainly, for each $f\in\mathscr{D}(\mathcal{M}(X,\mathcal{A}))$, $f\sim 1_{X\setminus Z(f)}$. We choose a subset $V$ of $\mathscr{D}(\mathcal{M}(X,\mathcal{A}))$ obeying the following conditions:
	\begin{enumerate}
		\item every element of $V$ is of the form $1_A$ where $\mu(A)>0$ and $\mu(X\setminus A)>0$.
		\item $1_A,1_B$ are distinct elements in $V$ if and only if $\mu(A\triangle B)>0$.
	\end{enumerate}
	Then $V$ is a collection of distinct class representatives corresponding to the equivalence relation $\sim$ on $\mathscr{D}(\mathcal{M}(X,\mathcal{A}))$. Let $G_2$ be the induced subgraph of $\Gamma'_2(\mathcal{M}(X,\mathcal{A}))$ whose set of vertices is $V$. We first observe the following:
	\begin{obv}\label{Obv4.1}
		\begin{enumerate}
			\item $[1_A]$ is a stable set in $\Gamma'_2(\mathcal{M}(X,\mathcal{A}))$,  for each $1_A\in V$. (Follows from Lemma \ref{Lem3.13}(\ref{Lem3.13.1})).
			\item For any two distinct stable sets $[1_A],[1_B]$ in $\Gamma'_2(\mathcal{M}(X,\mathcal{A}))$, either $[1_A]\sqcup [1_B]$ is a stable set or $[1_A]\sqcup[1_B]$ forms a complete bipartite subgraph of $\Gamma'_2(\mathcal{M}(X,\mathcal{A}))$ (Follows from Lemma \ref{Lem3.13}(\ref{Lem3.13.2})).
		\end{enumerate}
	\end{obv}
	Thus $\Gamma'_2(\mathcal{M}(X,\mathcal{A}))$ is a $|V|$-partite graph such that for any two stable sets $V_1,V_2$, either $V_1\sqcup V_2$ is a stable set or $V_1\sqcup V_2$ is a complete bipartite subgraph of $\Gamma'_2(\mathcal{M}(X,\mathcal{A}))$. Therefore, $G_2$ is a subgraph of $\Gamma'_2(\mathcal{M}(X,\mathcal{A}))$ where two vertices $1_A$, $1_B$ are adjacent if and only if the corresponding two stable sets $[1_A]$ and $[1_B]$ make a complete bipartite subgraph. We jot down in the next theorem some of the features highlighting the behaviour of $G_2$ as a graph and observe the analogy with its parent graph:
	\begin{thm}
		\hspace{3cm}
		\begin{enumerate}
			\item Let $1_A,1_B\in V$. There is a vertex in $V$ adjacent to both $1_A,1_B$ in $G_2$ if and only if $\mu(A\cap B)>0$.
			\item The distance between two vertices $1_A,1_B$ in $G_2$ is given by 
			$$d_{G_2}(1_A,1_B)=\begin{cases}
				1&\text{ if }\mu(X\setminus A\cap X\setminus B)=0\\
				2&\text{ if }\mu(X\setminus A\cap X\setminus B)>0\text{ and }\mu(A\cap B)>0\\
				3&\text{ if }\mu(X\setminus A\cap X\setminus B)>0\text{ and }\mu(A\cap B)=0
			\end{cases}$$
			\item $G_2$ is a connected graph.
			\item $|V|=2$ i.e., $G_2=K_2$ if and only if $X$ can be partitioned into two atoms.
			\item The diameter of $G_2$ is $diam(G_2)=\begin{cases}
				2&\text{ if }|V|=2\\
				3&\text{ otherwise}
			\end{cases}$
			\item The eccentricity of $1_A\in V$ in $G_2$ is $ecc_{G_2}(1_A)=\begin{cases}
				2&\text{ if }X\setminus A\text{ is an atom}\\
				3&\text{ otherwise}
			\end{cases}$
			\item A vertex $1_A\in V$ is a vertex of a triangle in $G_2$ if and only $A$ is not an atom.
			\item The girth of $G_2$ is $=\begin{cases}
				\infty&\text{ if }|V|=2\\
				3&\text{ otherwise}
			\end{cases}$
			\item $G_2$ is triangulated if and only if $\mu$ is a non-atomic measure.
			\item $G_2$ is not hypertriangulated.
			\item $1_A,1_B\in V$ are orthogonal in $G_2$ if and only if $\mu(A\cap B)=0=\mu(X\setminus A\cap X\setminus B)$.
			\item For every $1_A\in V$ there exists unique $1_B\in V$ such that $1_A\perp_{G_2} 1_B$ [uniqueness follows from Lemma \ref{Lem2.17}].
			\item $G_2$ is uniquely complemented.
		\end{enumerate}
	\end{thm}
	To compare the clique number of $\Gamma'_2(\mathcal{M}(X,\mathcal{A}))$ with that of its subgraph $G_2$, we need the following lemma.
	\begin{lem}\label{Lem4.1}
		If $M$ is a complete subgraph of $\Gamma'_2(\mathcal{M}(X,\mathcal{A}))$, then there exists a complete subgraph $M'$ of $G_2$ such that $|M|=|M'|$.
	\end{lem}
	\begin{proof}
		Let $M$ be a complete subgraph of $\Gamma'_2$. For each vertex $f$ in $M$, there is a vertex $1_A\in V$ such that $f\sim 1_A$. Let $f,g$ be distinct vertices in $M$ and $1_A,1_B\in V$ are such that $1_A\sim f$ and $1_B\sim g$. Since $M$ is a complete subgraph of $\Gamma'_2$, $f,g$ are adjacent in $\Gamma'_2$. By Lemma \ref{Lem3.13}(\ref{Lem3.13.3}), $1_A,1_B$ are adjacent. Consequently, $1_A,1_B$ are distinct vertices in $G_2$ and they are adjacent in $G_2$. Let $M'$ be the subgraph of $G_2$ whose vertex set is $\{1_A\in V:1_A\sim f\text{ for some vertex }f\text{ in }M\}$. Clearly, $M'$ is a complete subgraph of $G_2$ and $|M|=|M'|$.
	\end{proof}
	Since $G_2$ is a subgraph of $\Gamma'_2$, $cl(G_2)\leq cl(\Gamma'_2)$. So, using Lemma \ref{Lem4.1}, we get the following:
	\begin{thm}\label{Th4.3}
		The clique number of $G_2$ and $\Gamma'_2(\mathcal{M}(X,\mathcal{A}))$ are the same.
	\end{thm}
	As $\Gamma'_2(\mathcal{M}(X,\mathcal{A}))$ is a $|V|$-partite graph, $\Gamma'_2(\mathcal{M}(X,\mathcal{A}))$ is $|V|$-colorable and consequently $\chi(\Gamma'_2(\mathcal{M}(X,\mathcal{A})))$$\leq |V|=|G_2|$. We conclude that the chromatic number of both $\Gamma'_2(\mathcal{M}(X,\mathcal{A}))$ and $G_2$ are equal from the following lemma.
	\begin{lem}\label{Lem4.4}
		Let $\alpha$ be a cardinal number. If $G_2$ is $\alpha$-colorable, then we can color $\Gamma'_2(\mathcal{M}(X,\mathcal{A}))$ by using $\alpha$-many colors.
	\end{lem}
	\begin{proof}
		For each $f\in\mathscr{D}$, there exists $1_A\in V$ such that $f\sim 1_A$. Since $G_2$ is $\alpha$-colorable, $1_A$ already gets a color. We color each $f\in[1_A]$ by the color of $1_A\in V$. As a consequence, each $f\in \mathscr{D}$ gets a color from the available $\alpha$ many colors . It only remains to show that this coloring on $\Gamma'_2$ is consistent. Let $f,g\in\mathscr{D}$ be colored by the same color, say the color of $1_A$. By our method of coloring of $\Gamma'_2$, $f\sim 1_A$  and $g\sim 1_A$. Since $[1_A]$ is a stable set in $\Gamma'_2$, $f,g$ are non-adjacent in $\Gamma'_2$. Consequently, $\Gamma'_2$ is $\alpha$-colorable.
	\end{proof}
	Since $G_2$ is a subgraph of $\Gamma'_2$, $\chi(G_2)\leq\chi(\Gamma'_2)$. From Lemma \ref{Lem4.4}, we get the following result.
	\begin{thm}\label{Th4.5}
		The chromatic number of $G_2$ and $\Gamma'_2(\mathcal{M}(X,\mathcal{A}))$ are equal.
	\end{thm}
	
	\begin{thm}\label{Th4.6}
		$dt(G_2)\leq dt(\Gamma'_2(\mathcal{M}(X,\mathcal{A})))$.
	\end{thm}
	\begin{proof}
		Let $D\subset\mathscr{D}$ be a dominating set in $\Gamma'_2$. Let $V'=\{1_A\in V:1_A\sim f\text{ for some }f\in D\}$. Clearly, $|V'|\leq |D|$ and $D\cap V\subset V'$. To show that $V'$ is a dominating set in $G_2$, let $1_B\in V\setminus V'$. If $1_B\in D$, then $1_B\in V'$, which is not. So, $1_B\in\mathscr{D}\setminus D$. Since $D$ is a dominating set in $\Gamma'_2$, there exists $f\in D$ such that $f,1_B$ are adjacent in $\Gamma'_2$. Let $1_A\in V'$ such that $1_A\sim f$. By Observation \ref{Obv4.1}, $1_A,1_B$ are adjacent in $\Gamma'_2$ and hence they are adjacent in $G_2$. Therefore, $V'$ is a dominating set in $G_2$. Now $dt(G_2)\leq |V'|\leq |D|$ and this holds for every dominating set $D$ in $\Gamma'_2$. Consequently, $dt(G_2)\leq dt(\Gamma'_2)$.
	\end{proof}
	\begin{lem}\label{Lem4.7}
		Let $D\subset\mathscr{D}$.
		\begin{enumerate}
			\item Every total dominating set in $G_2$ is also a total dominating set in $\Gamma'_2(\mathcal{M}(X,\mathcal{A}))$.\label{Lem4.7.1}
			\item If $D$ is a total dominating set in $\Gamma'_2(\mathcal{M}(X,\mathcal{A}))$, then there exists a total dominating set $D'$ in $G_2$ such that $|D| \geq|D'|$.\label{Lem4.7.2}
		\end{enumerate}
	\end{lem}
	\begin{proof}
		\hspace{3cm}
		\begin{enumerate}
			\item Let $V'$ be a total dominating set in $G_2$ and $f\in\mathscr{D}$. Let $1_A\in V$ such that $f\sim 1_A$. Since $V'$ is a total dominating set in $G_2$, there exists $1_B\in V'$ such that $1_A,1_B$ are adjacent in $G_2$. By Lemma \ref{Lem3.13}(\ref{Lem3.13.3}), $f,1_B$ are adjacent in $\Gamma'_2$. Thus $V'$ be a total dominating set in $\Gamma'_2$.
			\item Let $D$ be a total dominating set in $\Gamma'_2$ and $D'=\{1_A\in V:1_A\sim f\text{ for some }f\in D\}$. Clearly, $|D'|\leq |D|$ and $D\cap V\subset D'$. Let $1_A\in V$. Then $1_A\in\mathscr{D}$ and hence there exists $f\in D$ such that $f,1_A$ are adjacent in $\Gamma'_2$. Consequently, $1_A$ is adjacent to $1_B$ in $G_2$ where $f\sim 1_B\in D'$. Hence, $D'$ is a a total dominating set in $G_2$.
		\end{enumerate}
	\end{proof}
	From  Lemma \ref{Lem4.7} the following result is immediate.
	\begin{thm}
		The total dominating number of $G_2$ and $\Gamma'_2(\mathcal{M}(X,\mathcal{A}))$ are equal.
	\end{thm}
	
	For any two subsets $A, B$ of $X$, as $A \triangle B = (X\setminus A) \triangle (X\setminus B)$ holds, we get a similar induced subgraph $G$ of the zero-divisor graph $\Gamma(\mathcal{M}(X,\mathcal{A}))$, arising from the same equivalence relation $\sim$ on $\mathscr{D}$ and considering the same set of vertices $V$ as in the case of comaximal graph. In this case too, each equivalence class $[1_A]$ is a stable set in $\Gamma(\mathcal{M}(X,\mathcal{A}))$ and for any two distinct equivalence classes $[1_A]$ and $[1_B]$, either $[1_A]\sqcup [1_B]$ is a stable set in $\Gamma(\mathcal{M}(X,\mathcal{A}))$ or they make a complete bipartite subgraph of $\Gamma(\mathcal{M}(X,\mathcal{A}))$. Hence, $\Gamma(\mathcal{M}(X,\mathcal{A}))$ is also a $|V|$-partite graph. This inherent similarity between the induced subgraphs $G$ and $G_2$ leads to a more stronger conclusion, as seen in the next proposition.
	
	\begin{thm}\label{Th5.2.1}
		$G$ and $G_2$ are graph isomorphic.
	\end{thm}
	\begin{proof}
		Let us define a function $\psi:V\to V$ as follows $\psi(1_A)=1_{\psi(A)}$ where $1_{\psi(A)}\sim 1_{X\setminus A}$. Since for $1_A\in V\implies \mu(A),\mu(X\setminus A)>0$, there exists $1_B\in V$ such that $1_B\sim 1_{X\setminus A}$. Also if $1_A$ and $1_B$ are two distinct elements in $V$, then $\mu(A\triangle B)>0\implies\mu(X\setminus A\triangle X\setminus B)>0$. As $1_{\psi(A)}\sim 1_{X\setminus A}$ and $1_{\psi(B)}\sim 1_{X\setminus B}$, $\mu(\psi(A)\triangle \psi(B))>0$ i.e., $1_{\psi(A)}$ and $1_{\psi(B)}$ are distinct in $V$. Consequently, $\psi$ is well-defined and injective. It is also clear that, for each $1_B\in V$, $\psi(1_{A})=1_B$, where $1_A\sim 1_{X\setminus B}$. Thus $\psi$ is surjective and hence $\psi$ is a bijection on $V$. We now claim that, $\psi$ preserves the adjacency relation between $G$ and $G_2$. Let $1_A,1_B\in V$ be adjacent in $G$; i.e., $\mu(A\cap B)=0$. It suffices to show that $1_{\psi(A)},1_{\psi(B)}$ are adjacent in $G_2$. We have $1_{\psi(A)}\sim 1_{X\setminus A}$ and $1_{\psi(B)}\sim 1_{X\setminus B}$. Since $\mu(A\cap B)=0$, $1_{X\setminus A}, 1_{X\setminus B}$ are adjacent in $G_2$. By Lemma \ref{Lem3.13}(\ref{Lem3.13.3}), $\psi(1_A),\psi(1_B)$ are adjacent in $G_2$. Similarly if $1_A,1_B$ are adjacent in $G_2$, then $\psi^{-1}(1_A),\psi^{-1}(1_B)$ are adjacent in $G$. Hence $\psi$ is an isomorphism between the two graphs $G$ and $G_2$.
	\end{proof}
	Theorem \ref{Th5.2.1} prompted us to suspect that $\Gamma(\mathcal{M}(X,\mathcal{A}))$ and $\Gamma'_2(\mathcal{M}(X,\mathcal{A}))$ are isomorphic. But, in the last section of this paper we could find an example of a measure space on which these two graphs are not isomorphic. However, we establish a sufficient condition for these two graphs to be isomorphic in general, as given below.\\
	If $|[1_A]|=|[1_{X\setminus A}]|$ for each $1_A\in V$, then there exists a bijection between $[1_A]$ and $[1_{\psi(A)}]$, because $[1_{X\setminus A}]=[1_{\psi(A)}]$. For each $1_A\in V$, let $\phi_A:[1_A]\to[1_{\psi(A)}]$ be a bijection. Define $\phi:\mathscr{D} (\mathcal{M}(X,\mathcal{A}))\to\mathscr{D}(\mathcal{M}(X,\mathcal{A}))$ by $\phi(f)=\phi_A(f)$ whenever $f\sim 1_A$. Since each $\phi_A$ is a bijection and $V$ is the collection of the class representatives under the equivalence relation $\sim$, $\phi$ is a bijective map. Let $f,g\in \mathscr{D}(\mathcal{M}(X,\mathcal{A}))$ be adjacent in $\Gamma(\mathcal{M}(X,\mathcal{A}))$. Let $1_A,1_B\in V$ be such that $f\sim 1_A$ and $g\sim 1_B$. Since $f,g$ are adjacent, it can be proved easily that $1_A,1_B$ are adjacent in $G$. Therefore $\psi(1_A), \psi(1_B)$ are adjacent in $G_2$. By the definition of $\phi$, $\phi(f)\sim\psi(1_A)$ and $\phi(g)\sim\psi(1_B)$. Therefore $\phi(f),\phi(g)$ are adjacent in $\Gamma'_2(\mathcal{M}(X,\mathcal{A}))$. Similarly if $f,g$ are adjacent in $\Gamma'_2(\mathcal{M}(X,\mathcal{A}))$, then $\phi^{-1}(f),\phi^{-1}(g)$ are adjacent in $\Gamma(\mathcal{M}(X,\mathcal{A}))$ (follows from Lemma \ref{Lem3.13}(\ref{Lem3.13.3})). Hence, $\Gamma(\mathcal{M}(X,\mathcal{A}))$ and $\Gamma'_2(\mathcal{M}(X,\mathcal{A}))$ are graph isomorphic.\\ 
	We record this fact in the form of a theorem.
	\begin{thm}\label{Iso}
		If $|[1_A]|=|[1_{X\setminus A}]|$ for each $1_A\in V$, then $\Gamma(\mathcal{M}(X,\mathcal{A}))$ and $\Gamma'_2(\mathcal{M}(X,\mathcal{A}))$ are graph isomorphic.
	\end{thm}

	\section{\bf The annihilator graph $AG(\mathcal{M}(X,\mathcal{A}))$ of $\mathcal{M}(X,\mathcal{A})$}\label{Sec5}
	We redefine the annihilator graph of the ring $\mathcal{M}(X,\mathcal{A})$ as follows:
	\begin{defn}
		The annihilator graph $AG(\mathcal{M}(X,\mathcal{A}))$ (in short $AG$) of $\mathcal{M}(X,\mathcal{A})$ is a simple graph whose set of vertices is $\mathscr{D}(\mathcal{M}(X,\mathcal{A}))$ and the adjacency relation is given by the following rule : $f,g$ are adjacent if and only if $ann(f)\cup ann(g)\subsetneqq ann(f.g)$.
	\end{defn} 
	It is quite clear that for any two $f, g \in \mathcal{M}(X, \mathcal{A})$, $ann(f) \cup ann(g) \subseteq ann(f.g)$. So, two vertices $f, g$ are non-adjacent when and only when $ann(f) \cup ann(g) = ann(f.g)$. We first observe that the adjacency of two vertices of $AG(\mathcal{M}(X,\mathcal{A}))$ can be determined by measuring the differences between their corresponding zero sets.
	
	\begin{thm}\label{Th0.5}
		$f,g \in \mathscr{D}(\mathcal{M}(X,\mathcal{A}))$ are adjacent in $AG(\mathcal{M}(X,\mathcal{A}))$ if and only if $\mu(Z(f)\setminus Z(g))>0$ and $\mu(Z(g)\setminus Z(f))>0$.
	\end{thm}
	\begin{proof}
		Let $f,g$ be not adjacent in $AG$. Then $ann(f)\cup ann(g)=ann(f.g)$. Since union of two ideals is an ideal, either $ann(f)\subset ann(g)$ or $ann(g)\subset ann(f)$. By Theorem \ref{Lem0.4}, either $\mu(Z(f)\setminus Z(g))=0$ or $\mu(Z(g)\setminus Z(f))=0$. Conversely, for definiteness sake, let $\mu(Z(f)\setminus Z(g))=0$. Then by Theorem \ref{Lem0.4}, $ann(f)\subset ann(g)$. Let $h\in ann(f.g)$. Then $\mu(X \setminus Z(h) \cap X \setminus Z(f) \cap X \setminus Z(g)) = 0$. Now,  $X\setminus Z(h)\cap X\setminus Z(g) = (X\setminus Z(h)\cap X\setminus Z(g)  \cap Z(f))\cup (X\setminus Z(h)\cap X\setminus Z(g) \cap (X \setminus Z(f)))\subset (Z(f)\setminus Z(g))\cup (X\setminus Z(h)\cap X\setminus Z(g) \cap X \setminus Z(f))$ implies $\mu(X\setminus Z(h)\cap X\setminus Z(g))=0$. Hence, $h \in ann(g)$. i.e., $ann(g) = ann(f) \cup ann(g)=ann(f.g)$ which proves $f, g$ are non-adjacent in AG. 
	\end{proof}
	The next theorem sets conditions for the existence of a third vertex, adjacent to a given pair of vertices in $AG(\mathcal{M}(X,\mathcal{A}))$.
	\begin{thm}\label{Th5.1}
		Let $f,g\in\mathscr{D}(\mathcal{M}(X,\mathcal{A}))$. 
		\begin{enumerate}
			\item If $\mu(X\setminus Z(f)\cap X\setminus Z(g))>0$ or $\mu(Z(f)\cap Z(g))>0$, then there is a vertex adjacent to both $f$ and $g$ in $AG(\mathcal{M}(X,\mathcal{A}))$.
			\item Let $\mu(X\setminus Z(f)\cap X\setminus Z(g))=0=\mu(Z(f)\cap Z(g))$. Then there exists a vertex adjacent to both $f$ and $g$ if and only if $Z(f)$ and $Z(g)$ are not atoms.
		\end{enumerate}
	\end{thm}
	\begin{proof}
		It is easy to check that $\mu(X\setminus Z(f)\cap X\setminus Z(g))>0$ or $\mu(Z(f)\cap Z(g))>0$ imply that $1_{Z(f)\cup Z(g)}$ or $1_{Z(f)\cap Z(g)}$ is adjacent to both $f$, $g$ in $AG$ respectively. Let $\mu(X\setminus Z(f)\cap X\setminus Z(g))=0=\mu(Z(f)\cap Z(g))$. Without loss of generality let, $Z(f)=X\setminus Z(g)$. If both $Z(f)$ and $Z(g)$ are not atoms, then there exist $A_1,A_2,B_1,B_2\in\mathcal{A}$ each with positive measure such that $Z(f)=A_1\sqcup A_2$ and $Z(g)=B_1\sqcup B_2$. Consider $h= 1_{A_1\cup B_1}$. Clearly $h\in\mathscr{D}$ and $A_2\cup B_2\subset Z(h)$. Therefore, $\mu(Z(f)\setminus Z(h))>0$ and $\mu(Z(h)\setminus Z(f))>0$ i.e., $f,h$ are adjacent in $AG$. Similarly $g,h$ are adjacent in $AG$. Conversely let $h$ be adjacent to both $f,g$ in $AG$. Then $\mu(Z(f)\setminus Z(h))>0$ and $\mu(Z(h)\setminus Z(g))>0$. Now $Z(f)=(Z(f)\cap Z(h))\sqcup(Z(f)\cap X\setminus Z(h))= (Z(h)\setminus Z(g))\sqcup(Z(f)\setminus Z(h))$, as $Z(f)=X\setminus Z(g)$. Therefore, $Z(f)$ is not an atom and similarly, $Z(g)$ is also not an atom.
	\end{proof}
	\begin{cor}
		Let $f,g\in\mathscr{D}(\mathcal{M}(X,\mathcal{A}))$. Then $$d(f,g)=\begin{cases}
			1&\text{ if }f,g\text{ are adjacent}\\
			2&\text{ otherwise}
		\end{cases}$$
	\end{cor}
	\begin{cor}\label{cor4.5}
		$ecc(f)=2$ for all $f\in\mathscr{D}(\mathcal{M}(X,\mathcal{A}))$.
	\end{cor}
	\begin{cor}
		The diameter of $AG(\mathcal{M}(X,\mathcal{A}))$ is $2$.
	\end{cor}
	We have already seen that the vertex sets of the comaximal graph and the zero-divisor graph of $\mathcal{M}(X,\mathcal{A})$ are also $\mathscr{D}(\mathcal{M}(X,\mathcal{A}))$; i.e., the same as the vertex set of AG. If $f, g \in \mathscr{D}(\mathcal{M}(X,\mathcal{A}))$ are adjacent in $\Gamma'_2(\mathcal{M}(X,\mathcal{A}))$ then $\mu(Z(f)\cap Z(g))= 0$. So, $(Z(f)\setminus Z(g))\sqcup (Z(f)\cap Z(g)) = Z(f)\implies \mu(Z(f)\setminus Z(g))=\mu(Z(f))>0$. Similarly, $\mu(Z(g)\setminus Z(f)) > 0$. This indicates that $f, g$ are adjacent in $AG(\mathcal{M}(X,\mathcal{A}))$. By an analogous argument we also get that the vertices which are adjacent in $\Gamma(\mathcal{M}(X,\mathcal{A}))$, are adjacent in $AG(\mathcal{M}(X,\mathcal{A}))$. We record these observations in the form of a theorem :
	\begin{thm}\label{Th4.4}
		Both $\Gamma(\mathcal{M}(X,\mathcal{A}))$ and $\Gamma'_2(\mathcal{M}(X,\mathcal{A}))$ are subgraphs of $AG(\mathcal{M}(X,\mathcal{A}))$.
	\end{thm}
	\begin{thm}\label{Th5.5}
		$AG(\mathcal{M}(X,\mathcal{A}))$ is a complete bipartite graph if and only if $X$ is partitioned into two atoms.
	\end{thm}
	\begin{proof}
		Let $AG$ be complete bipartite by bipartion $V_1, V_2$. Since $V_1\neq\emptyset$ and $V_2\neq\emptyset$, fix $f\in V_1$ and $g\in V_2$. Since $AG$ is a complete bipartite graph, there does not exit any vertex adjacent to both $f,g$ and hence by Theorem \ref{Th5.1}, $\mu(X\setminus Z(f)\cap X\setminus Z(g))=0=\mu(Z(f)\cap Z(g))$ and either $Z(f)$ or $Z(g)$ is an atom. Suppose $Z(f)$ is an atom. We claim that $Z(g)$ is also an atom. If possible let $Z(g)=A_1\sqcup A_2$ where $A_1,A_2\in\mathcal{A}$ with $\mu(A_1),\mu(A_2)>0$. Consider $h=1_{X\setminus(Z(f)\cup A_1)}$. Then $h\in\mathscr{D}$ and $Z(h)=Z(f)\cup A_1$. Clearly, $h$ is adjacent to both $f,g$, which is a contradiction. Thus $Z(g)$ is an atom. Let $A=Z(f)\setminus Z(g)=Z(f)\setminus (Z(f)\cap Z(g))$ and $B=Z(g)\cup(X\setminus Z(f))=Z(g)\cup((X\setminus Z(f))\cap (X\setminus Z(g)))$. Thus $A,B$ are atoms and $X=A\sqcup B$. Conversely let $A,B$ be two atoms of $X$ such that $X=A\sqcup B$. Then by Lemma \ref{Lem3.14}, $\mathscr{D}$ is partitioned into two sets, say $V_1,V_2$, where $V_1=\{f\in\mathscr{D}: \mu(A\triangle Z(f))=0\}$ and $V_2=\{f\in\mathscr{D}:\mu(B\triangle Z(f))=0\}$. Let $f_1,f_2\in V_1$. Then $\mu(Z(f_1)\setminus A)=0= \mu(A\setminus Z(f_2))$. Now $Z(f_1)\subset A\sqcup Z(f_1)\setminus A\implies Z(f_1)\setminus Z(f_2)\subset A\setminus Z(f_2)\sqcup Z(f_1)\setminus A\implies\mu(Z(f_1)\setminus Z(f_2))=0$ i.e., $f_1,f_2$ are non-adjacent in $AG$. Consequently $V_1$ is a stable set in $AG$ and similarly $V_2$ is a stable set. Now let $f\in V_1$ and $g\in V_2$. Then $\mu(Z(f)\setminus A)=0=\mu(Z(g)\setminus B)$. Now, $Z(f)\subset A\sqcup Z(f)\setminus A$ and $Z(g)\subset B\sqcup Z(g)\setminus B$. Therefore, $Z(f)\cap Z(g)\subset (A\sqcup Z(f)\setminus A)\cap(B\sqcup Z(g)\setminus B)\subset (A\cap B)\cup(Z(f)\setminus A)\cup(Z(g)\setminus B)\implies\mu(Z(f)\cap Z(g))=0$. i.e., $f,g$ are adjacent in $\Gamma'_2$. By Theorem \ref{Th4.4}, $f,g$ are adjacent in $AG$ and hence $AG$ is a complete bipartite graph.
	\end{proof}
	Combining Theorem~\ref{CompBip} and Theorem~\ref{Th5.5}, we get the following result.
	\begin{cor}
		The following statements are equivalent:
		\begin{enumerate}
			\item $\Gamma(\mathcal{M}(X,\mathcal{A}))$ is a complete bipartite graph.
			\item $\Gamma'_2(\mathcal{M}(X,\mathcal{A}))$ is a complete bipartite graph.
			\item $AG(\mathcal{M}(X,\mathcal{A}))$ is a complete bipartite graph.
			\item $X$ is partitioned into two atoms.
		\end{enumerate}
	\end{cor}
	
	\begin{thm}\label{Th5.10}
		The following statements are equivalent:
		\begin{enumerate}
			\item $\Gamma(\mathcal{M}(X,\mathcal{A}))= AG(\mathcal{M}(X,\mathcal{A}))$.
			\item $\Gamma'_2(\mathcal{M}(X,\mathcal{A}))= AG(\mathcal{M}(X,\mathcal{A}))$.
			\item $X$ is partitioned into two atoms.
		\end{enumerate}
	\end{thm}
	\begin{proof}
		If $X$ is partitioned into two atoms, then all the three graphs are complete bipartite with same bipartion and hence they are equal. Suppose $X$ can not be partitioned into two atoms. Consider an $A\in\mathcal{A}$ such that $\mu(A),\mu(X\setminus A)>0$. Then either $A$ or $X\setminus A$ is not an atom. Suppose $X\setminus A$ is not an atom. Then there exist $A_1,A_2\in\mathcal{A}$ with $\mu(A_1),\mu(A_2)>0$ such that $X\setminus A=A_1\sqcup A_2$. Let $f_1=1_{A\sqcup A_1}\in\mathscr{D}$ and $f_2=1_{A\sqcup A_2}\in \mathscr{D}$. Then $X\setminus Z(f_1)\cap X\setminus Z(f_2)=A$ and $Z(f_1)\cap Z(f_2)=\emptyset$. Therefore, $f_1,f_2$ are not adjacent in $\Gamma$, but they are adjacent in $\Gamma'_2$ and hence they are adjacent in $AG$. Therefore, $\Gamma\neq AG$. Again let $g_1=1_{A_1}\in\mathscr{D}$ and $g_2=1_{A_2}\in\mathscr{D}$. Then $Z(g_1)\cap Z(g_2)=A$ and $X\setminus Z(g_1)\cap X\setminus Z(g_2)=\emptyset$. Therefore, $g_1,g_2$ are not adjacent in $\Gamma'_2$, but are adjacent in $\Gamma$ and hence they are adjacent in $AG$. Therefore, $\Gamma'_2\neq AG$.
	\end{proof}
	\begin{thm}
		A vertex $f$ in $AG(\mathcal{M}(X,\mathcal{A}))$ is a vertex of a triangle if and only if either $Z(f)$ or $X\setminus Z(f)$ is not an atom.
	\end{thm}
	\begin{proof}
		If $X\setminus Z(f)$ is not an atom, then by Theorem \ref{Th2.9}, $f$ is a vertex of triangle in $\Gamma'_2$. By Theorem \ref{Th4.4}, $f$ is a vertex of a triangle in $AG$. Again if $Z(f)$ is not an atom, then there exists $A_1,A_2\in\mathcal{A}$ with $\mu(A_1),\mu(A_2)>0$ such that $Z(f)=A_1\sqcup A_2$. Clearly $1_{A_1},1_{A_2}\in\mathscr{D}$ and $f.1_{A_1}\equiv 0\equiv 1_{A_1}.1_{A_2}\equiv 1_{A_2}.f$ on $X$. i.e., $f-1_{A_1}-1_{A_2}-f$ is a triangle in $\Gamma$. By Theorem \ref{Th4.4}, $f-1_{A_1}-1_{A_2}-f$ is a triangle in $AG$. Let both $Z(f)$ and $X\setminus Z(f)$ be atoms. Then by Theorem \ref{Th5.5}, $AG$ is complete bipartite. Hence $f$ can not be a vertex of a triangle.
	\end{proof}
	\begin{cor}
		The girth of $AG(\mathcal{M}(X,\mathcal{A}))$ is given by $$gr(AG(\mathcal{M}(X,\mathcal{A})))=\begin{cases}
			4&\text{ if }X\text{ is partitioned into two atoms}\\
			3&\text{ otherwise}
		\end{cases}$$.
	\end{cor}
	\begin{cor}
		$AG(\mathcal{M}(X,\mathcal{A}))$ is triangulated if and only if $X$ is not partitioned into two atoms.
	\end{cor}
	The next result directly follows from Theorem \ref{Th5.1}.
	\begin{thm}\label{Th5.4}
		Let $f,g\in\mathscr{D}(\mathcal{M}(X,\mathcal{A}))$. Then $f\perp g$ in $AG(\mathcal{M}(X,\mathcal{A}))$ if and only if $\mu(X\setminus Z(f)\cap X\setminus Z(g))=0=\mu(Z(f)\cap Z(g))$ and either $Z(f)$ or $Z(g)$ is an atom.
	\end{thm}
	\begin{cor}\label{Cor5.5}
		An edge $f-g$ in $AG(\mathcal{M}(X,\mathcal{A}))$ is an edge of a triangle if and only if $f\not\perp g$.
	\end{cor}
	If $X$ can be partitioned into two atoms, then $AG(\mathcal{M}(X,\mathcal{A}))$ is never hypertriangulated. On the contrary, if $X$ can not be partitioned into two atoms, then for any $A\in\mathcal{A}$ with $\mu(A)>0$ and $\mu(X\setminus A)>0$, either $A$ or $X\setminus A$ is not an atom. This condition helps in determining exactly when the graph $AG(\mathcal{M}(X,\mathcal{A}))$ is hypertriangulated.
	\begin{thm}\label{Th5.16}
		$AG(\mathcal{M}(X,\mathcal{A}))$ is hypertriangulated if and only if $\mu$ is non-atomic.
	\end{thm}
	\begin{proof}
		If $\mu$ is non-atomic, then for any pair of vertices $f,g$ in $AG$, $f\not\perp g$ (by Theorem \ref{Th5.4}), because $Z(f)$ and $Z(g)$ are not atoms. Consequently, by Corollary \ref{Cor5.5}, $AG$ is hypertriangulated. Conversely, let $A\in\mathcal{A}$ be an atom. Then, $\mu(X\setminus A)>0$, as $X$ is not an atom. Thus, $1_A, 1_{X\setminus A}\in\mathscr{D}$. By Theorem \ref{Th5.4}, $1_A\perp 1_{X\setminus A}$. Therefore, $1_A-1_{X\setminus A}$ is not an edge of a triangle in $AG$.
	\end{proof}
	\begin{cor}
		$AG(\mathcal{M}(X,\mathcal{A}))$ is hypertriangulated if and only if $\Gamma'_2(\mathcal{M}(X,\mathcal{A}))$ is triangulated.
	\end{cor}
	
	\begin{thm}
		Let $f,g\in\mathscr{D}(\mathcal{M}(X,\mathcal{A}))$. In $AG(\mathcal{M}(X,\mathcal{A}))$,
		$$c(f,g)=\begin{cases}
			3&f,g\text{ are adjacent and } f\not\perp g\\
			4&\text{ either }f,g\text{ are not adjacent or }f\perp g
		\end{cases}$$
	\end{thm}
	\begin{proof}
		Let $f,g\in\mathscr{D}$. If $f,g$ are adjacent and $f\not\perp g$, then there exists $h\in\mathscr{D}$ adjacent to both $f,g$; i.e., $f-g-h-f$ is a triangle in $AG$. Thus $c(f,g)=3$. If $f,g$ are not adjacent in $AG$, then they are not adjacent in $\Gamma'_2$ also, because $\Gamma'_2$ is a subgraph of $AG$. So, $\mu(Z(f)\cap Z(g))>0$. By Theorem \ref{Th5.1}, there exists $h\in\mathscr{D}$ adjacent to both $f,g$. Thus $f-h-g-2.h-f$ is a square in $AG\implies c(f,g)\leq 4$. Since $f,g$ are not adjacent in $AG$, $c(f,g)=4$. If $f\perp g$, then $f-g-2.f-2.g-f$ is a square in $AG$. Hence, $c(f,g)\leq 4$. Since $f\perp g$, no vertex is adjacent to both $f,g$; i.e., $c(f,g)>3$. Therefore, $c(f,g)=4$. 
	\end{proof}
	The visual representation of the above Theorem is exhibited as follows:\\
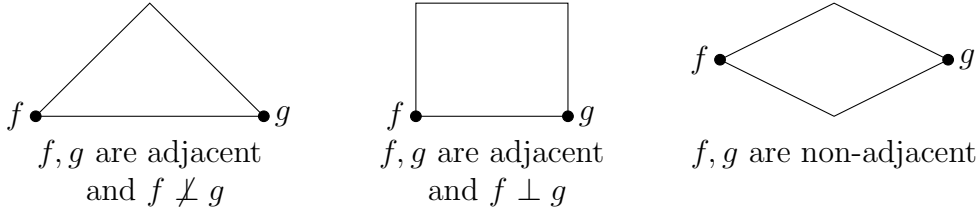
	\begin{center}
		\begin{tikzpicture}
			\draw (0,0)--(1.5,1.5)--(3,0)--(0,0) (5,0)--(5,1.5)--(7,1.5)--(7,0)--(5,0) (9,0.75)--(10.5,1.5)--(12,0.75)--(10.5,0)--(9,0.75);
			\filldraw[black] (0,0) circle (2pt) node[anchor=east]{$f$};
			\filldraw[black] (3,0) circle (2pt) node[anchor=west]{$g$};
			\filldraw[black] (5,0) circle (2pt) node[anchor=east]{$f$};
			\filldraw[black] (7,0) circle (2pt) node[anchor=west]{$g$};
			\filldraw[black] (9,0.75) circle (2pt) node[anchor=east]{$f$};
			\filldraw[black] (12,0.75) circle (2pt) node[anchor=west]{$g$};
			\node (start) at (1.5,-0.5) {$f,g\text{ are adjacent}$};
			\node (start) at (1.5,-1) {$\text{ and }f\not\perp g$};
			\node (start) at (6,-0.5) {$f,g\text{ are adjacent}$};
			\node (start) at (6,-1) {$\text{ and }f\perp g$};
			\node (start) at (10.5,-0.5) {$f,g\text{ are non-adjacent}$};
		\end{tikzpicture}
	\end{center}
	
	\begin{thm}\label{Th5.15}
		$f\in\mathscr{D}(\mathcal{M}(X,\mathcal{A}))$ has an orthogonal complement if and only if either $Z(f)$ or $X\setminus Z(f)$ is an atom.
	\end{thm}
	\begin{proof}
		If one of $Z(f)$ and $ X\setminus Z(f)$ is an atom, then by Theorem \ref{Th5.4}, $f\perp 1_{Z(f)}$ in $\text{AG}$. Conversely let $f\perp g$ for some $g\in\mathscr{D}$. Then $\mu(Z(f)\cap Z(g))=0=\mu(X\setminus Z(f)\cap X\setminus Z(g))$ and either $Z(f)$ or $Z(g)$ is an atom. If $Z(g)$ is an atom, then $X\setminus Z(f)$ is an atom, because $X\setminus Z(f)= [X\setminus Z(f)\cap X\setminus Z(g)]\cup Z(g)\setminus[Z(f)\cap Z(g)]$. Therefore either $Z(f)$ or $X\setminus Z(f)$ is an atom.
	\end{proof}
	
	\begin{lem}\label{Lem5.16}
		If $X$ is partitioned into three atoms, then for each $f\in\mathscr{D}(\mathcal{M}(X,\mathcal{A}))$, either $Z(f)$ or $X\setminus Z(f)$ is an atom.
	\end{lem}
	\begin{proof}
		Let $X=A\sqcup B\sqcup C$, where $A,B,C$ are atoms. If possible let there be some $f\in\mathscr{D}$ such that both $Z(f), X\setminus Z(f)$ are not atoms. So there are $A_1, A_2, B_1, B_2 \in \mathcal{A}$, such that $Z(f)=A_1\sqcup A_2$, $X\setminus Z(f)=B_1 \sqcup B_2$ and $\mu(A_1),\mu(A_2),\mu(B_1),\mu(B_2)>0$. Then $\mathscr{B}=\{A,B,C\}$ and $\mathscr{B}'=\{A_1,A_2,B_1,B_2\}$ constitute partitions of $X$ by sets with positive measure. Hence there exist $E\in\mathscr{B}$ and $F_1,F_2\in\mathscr{B}'$ such that $\mu(E\cap F_1),\mu(E\cap F_2)>0$. Now $E=(E\cap F_1)\sqcup(E\setminus F_1)$. Since, $F_1,F_2$ are disjoint, $E\cap F_2\subset E\setminus F_1$ and hence $\mu(E\setminus F_1)>0$. This contradicts that $E$ is an atom.
	\end{proof}
	\begin{thm}\label{Th5.20}
		$AG(\mathcal{M}(X,\mathcal{A}))$ is complemented if and only if $X$ is partitioned into either two or three atoms.
	\end{thm}
	\begin{proof}
		We prove this in three cases.\\
		Case I: $X=A\sqcup B$, where $A,B$ are atoms.\\ 
		By Theorem \ref{Th5.5}, $AG$ is a complete bipartite graph and hence $AG$ is complemented.\\
		Case II: $X=A\sqcup B\sqcup C$, where $A,B,C$ are atoms.\\
		Then by Lemma \ref{Lem5.16}, either $Z(f)$ or $X\setminus Z(f)$ is an atom for each $f\in\mathscr{D}$. By Theorem \ref{Th5.15}, each $f\in\mathscr{D}$ has an orthogonal complement. Consequently, $AG$ is complemented.\\
		Case III: Neither Case I nor Case II.\\
		Subcase I: $X$ contains no atom. Then for any pair $f,g\in\mathscr{D}$, $f\not\perp g$. Hence, $AG$ is not complemented.\\
		Subcase II: $X$ contains an atom, say $A\in\mathcal{A}$. Then $\mu(X\setminus A)>0$, for otherwise $X$ will be an atom. By our hypothesis, $X\setminus A$ is not an atom. Then $X\setminus A=B_1\sqcup B_2$, for some $B_1, B_2 \in \mathcal{A}$ with $\mu(B_1),\mu(B_2)>0$. Then again, one of $B_1,B_2$ is not an atom. Without loss of generality, let $B_1$ be not an atom. Then $B_1=C_1\sqcup C_2$,  for some $C_1, C_2 \in \mathcal{A}$ with $\mu(C_1),\mu(C_2)>0$. Consider $f\in\mathscr{D}$ such that $Z(f)=C_1\sqcup B_2$. Then $X\setminus Z(f)=C_2\sqcup A$. Therefore, both of $Z(f)$ and $X\setminus Z(f)$ are not atoms. By Theorem \ref{Th5.15}, $f\in\mathscr{D}$ has no orthogonal complement. Consequently, $AG$ is not complemented.
	\end{proof}
	\begin{lem}\label{Lem4.19}
		Let $f,g,h\in\mathscr{D}(\mathcal{M}(X,\mathcal{A}))$ be such that $f\perp g$ and $f\perp h$ in $AG$. Then $g,h$ are adjacent to the same set of vertices.
	\end{lem}
	\begin{proof}
		By Theorem \ref{Th5.4}, $\mu(Z(f)\cap Z(g))=0=\mu(X\setminus Z(f)\cap X\setminus Z(g))$ and $\mu(Z(f)\cap Z(h))=0=\mu(X\setminus Z(f)\cap X\setminus Z(h))$. i.e., $\mu(Z(g)\triangle (X\setminus Z(f)))=0=\mu( Z(h)\triangle (X\setminus Z(f)))$. This implies $\mu(Z(g)\triangle Z(h))=0$. Let $k\in\mathscr{D}$ be adjacent to $g$; i.e., $\mu(Z(k)\setminus Z(g))>0$ and $\mu(Z(g)\setminus Z(k))>0$. Now $Z(g)\setminus Z(k)= (Z(g)\cap Z(h)\cap X\setminus Z(k))\sqcup(Z(g)\setminus Z(h)\cap X\setminus Z(k))$. Since $\mu(Z(g)\setminus Z(h))=0$, $\mu(Z(g)\cap Z(h)\cap X\setminus Z(k))=\mu(Z(g)\setminus Z(k))>0\implies\mu(Z(h)\setminus Z(k))\geq\mu(Z(g)\cap Z(h)\cap X\setminus Z(k))>0$ i.e., $\mu(Z(h)\setminus Z(k))>0$. Similarly, $\mu(Z(k)\setminus Z(h))>0$. Thus $k$ is adjacent to $h$. Likewise if $k$ is adjacent to $h$, then it is adjacent to $g$ also. Hence $g,h$ are adjacent to the same set of vertices.
	\end{proof}
	As a consequence of this lemma, we get the following theorem.
	\begin{thm}
		$AG(\mathcal{M}(X,\mathcal{A}))$ is complemented if and only if $AG(\mathcal{M}(X,\mathcal{A}))$ is uniquely complemented.
	\end{thm}
	
	The fact that $\Gamma'_2(\mathcal{M}(X,\mathcal{A}))$ and $\Gamma(\mathcal{M}(X,\mathcal{A}))$ are subgraphs of $AG(\mathcal{M}(X,\mathcal{A}))$ helps us to determine the dominating sets in $AG(\mathcal{M}(X,\mathcal{A}))$, as seen in the next result.
	\begin{thm}\label{Th4.21}
		$dt(AG(\mathcal{M}(X,\mathcal{A})))=2$. 
	\end{thm}
	\begin{proof}
		Fix an $A\in\mathcal{A}$ with $\mu(A),\mu(X\setminus A)>0$. We claim that $\{1_A,1_{X\setminus A}\}$ is a dominating set in $AG$. Suppose $f\in\mathscr{D}$ is not adjacent to $1_{X\setminus A}$ in $AG$; i.e., either $\mu(Z(f)\setminus A)=0$ or $\mu(A\setminus Z(f))=0$. If $\mu(Z(f)\setminus A)=0$, then $\mu(Z(f)\cap X\setminus A)=0$. So $f,1_A$ are adjacent in $\Gamma'_2$ and hence by Theorem \ref{Th4.4}, $f,1_A$ are adjacent in $AG$. If $\mu(A\setminus Z(f))=0$, then $\mu(X\setminus Z(f)\cap A)=0$. Thus $f,1_A$ are adjacent in $\Gamma$. Again by Theorem \ref{Th4.4}, $f,1_A$ are adjacent in $AG$. Hence, in any case, $f$ is adjacent to $1_A$ in $AG$. Consequently, $\{1_A,1_{X\setminus A}\}$ is a dominating set in $AG$. Also for each $f\in\mathscr{D}$, $f,2.f$ are not adjacent in $AG$ which proves singleton sets are not dominating sets in $AG$. Hence $dt(AG)=2$.
	\end{proof}
	In the above theorem, we note that the dominating set $\{1_A,1_{X\setminus A}\}$ is a total dominating set in $AG(\mathcal{M}(X,\mathcal{A}))$. We record this observation in the following result.
	\begin{cor}\label{Cor4.22}
		$dt_t(AG(\mathcal{M}(X,\mathcal{A})))=2$.
	\end{cor}
	In Theorem~\ref{Th5.10}, we have seen that $\Gamma'_2(\mathcal{M}(X,\mathcal{A})) = AG(\mathcal{M}(X,\mathcal{A}))$  if and only if $X$ is partition into two atoms. We now show that these two graphs are identical if and only if they are isomorphic.
	
	\begin{thm}
		The graphs $\Gamma'_2(\mathcal{M}(X,\mathcal{A}))$ and $AG(\mathcal{M}(X,\mathcal{A}))$ are isomorphic if and only if $X$ is partitioned into two atoms.
	\end{thm}
	\begin{proof}
		If $X$ is partitioned into two atoms, then by Theorem~\ref{Th5.10}, these two graph are equal and so, they are isomorphic.\\
		If $X$ is not partitioned into two atoms then two cases may arise; either $X$ is partitioned into three atoms or $X$ is $X$ is not partitioned into three atoms. First we assume that $X$ is partitioned into three atoms, say $X=A\sqcup B\sqcup C$ where $A,B,C$ are atoms. If possible let $\psi$ be a graph isomorphism between the graphs $\Gamma'_2(\mathcal{M}(X,\mathcal{A}))$ and $AG(\mathcal{M}(X,\mathcal{A}))$. Now $1_A$ is a vertex in $\Gamma'_2(\mathcal{M}(X,\mathcal{A}))$ and $ecc(1_A)=3$ in $\Gamma'_2(\mathcal{M}(X,\mathcal{A}))$ (by Theorem~\ref{Th2.6}), because $Z(1_A)=B\sqcup C$, is not an atom. Since $\psi$ is a graph isomorphism, $ecc(\psi(1_A))=3$ in $AG(\mathcal{M}(X,\mathcal{A}))$, which contradicts the fact that $ecc(f)=2$ for every vertex $f$ in $AG(\mathcal{M}(X,\mathcal{A}))$ (Corollary \ref{cor4.5}). Therefore, there does not exist any graph isomorphism between these two graphs.\\
		Finally, let $X$ be neither partitioned into two atoms nor into three atoms. Then $\Gamma'_2(\mathcal{M}(X,\mathcal{A}))$ is a complemented graph (by Theorem~\ref{ComaxComp}), though $AG(\mathcal{M}(X,\mathcal{A}))$ is not complemented. Thus, they are not isomorphic as graphs.
	\end{proof}
	\begin{cor}
		$\Gamma(\mathcal{M}(X,\mathcal{A}))$ and $AG(\mathcal{M}(X,\mathcal{A}))$ are isomorphic if and only if $X$ is partitioned into two atoms.
	\end{cor}
	
	\section{\bf The weakly zero-divisor graph $W\Gamma(\mathcal{M}(X,\mathcal{A}))$ of $\mathcal{M}(X,\mathcal{A})$}\label{Sec6}
	The weakly zero-divisor graph $W\Gamma(\mathcal{M}(X,\mathcal{A}))$ (in short $W\Gamma$) of $\mathcal{M}(X,\mathcal{A})$, is defined as a graph with $\mathscr{D}(\mathcal{M}(X,\mathcal{A}))$ as the set of vertices and two vertices $f,g$ are adjacent if there exist $h_1\in ann(f)\cap\mathscr{D} (\mathcal{M}(X,\mathcal{A}))$ and $h_2\in ann(g)\cap \mathscr{D}(\mathcal{M}(X,\mathcal{A}))$ such that $h_1.h_2\equiv 0$ a.e. on $X$. \\
	
	The adjacency relation in $W\Gamma(\mathcal{M}(X,\mathcal{A}))$ is interpreted via the measure $\mu$ as follows:
	
	\begin{thm}\label{Th6.2}
		Let $f,g\in\mathscr{D}(\mathcal{M}(X,\mathcal{A}))$. In $W\Gamma(\mathcal{M}(X,\mathcal{A}))$,
		\begin{enumerate}
			\item if $\mu(Z(f)\triangle Z(g))>0$, then $f,g$ are adjacent.\label{6.2.1}
			\item if $\mu(Z(f)\triangle Z(g))=0$ and $Z(f)$ is not an atom, then $f,g$ are adjacent.
			\item if $\mu(Z(f)\triangle Z(g))=0$ and $Z(f)$ is an atom, then $f,g$ are not adjacent.\label{6.2.3}
		\end{enumerate}
	\end{thm}
	\begin{proof}
		\hspace*{3cm}
		\begin{enumerate}
			\item Since $\mu(Z(f)\triangle Z(g))>0$, either $\mu(Z(f)\setminus Z(g))>0$ or $\mu(Z(g)\setminus Z(f))>0$. Without loss of generality let $\mu(Z(f)\setminus Z(g))>0$. Consider $h_1=1_{Z(f)\setminus Z(g)}$ and $h_2=1_{Z(g)}$. Then $h_1,h_2\in\mathscr{D}$ and $h_1.f=0= h_2.g$ on $X$. So, $h_1\in ann(f)\cap\mathscr{D}$ and $h_2\in ann(g)\cap\mathscr{D}$. Also $h_1.h_2 = 0$ on $X$. Hence, $f,g$ are adjacent in $W\Gamma$.
			\item Since $Z(f)$ is not an atom, $Z(f)=A\sqcup B$ for some $A,B\in\mathcal{A}$ with $\mu(A),\mu(B)>0$. Clearly $1_A,1_B\in\mathscr{D}$ and $1_A.f= 0= 1_B.f$ on $X$ ; i.e., $1_A,1_B\in ann(f)$. Since $\mu(Z(f)\triangle Z(g))=0$,  by Corollary \ref{Lem6.1}, $ann(f)=ann(g)\implies 1_B\in ann(g)$. Also, $1_A.1_B= 0$ on $X\implies f,g$ are adjacent in $W\Gamma$.
			\item If possible let $f$ and $g$ be adjacent. Then there exist $h_1\in ann(f)\cap\mathscr{D}$ and $h_2\in ann(g)\cap \mathscr{D}$ such that $\mu(X\setminus Z(h_1)\cap X\setminus Z(h_2))=0$. Since $\mu(Z(f)\triangle Z(g))=0$, by Lemma \ref{Lem6.1}, $ann(f)=ann(g)$. Therefore $h_1,h_2\in ann(f)$; i.e., there exists a zero measurable set $E\in\mathcal{A}$ such that $X\setminus Z(h_1)\cap X\setminus E \subset Z(f)$ and $X\setminus Z(h_2) \cap X \setminus E\subset Z(f)$. Let $A=X\setminus Z(h_1)\cap X\setminus E$ and $B=Z(f)\setminus A$. Clearly, $A,B\in\mathcal{A}$ and $Z(f)=A\sqcup B$. Since $h_1\in\mathscr{D}$ and $\mu(E)=0$, $\mu(A)>0$. Again, $(X\setminus E\cap X\setminus Z(h_2))\setminus(X\setminus Z(h_1)\cap X\setminus Z(h_2))\subset B$. Since $h_2\in\mathscr{D}$ and $\mu(X\setminus Z(h_1)\cap X\setminus Z(h_2))=0=\mu(E)$, $\mu(B)>0$. This contradicts the hypothesis that $Z(f)$ is an atom. 
		\end{enumerate}
	\end{proof}
	From Theorem \ref{Th6.2}, it follows that $f,f$ are adjacent if and only if $Z(f)$ is not an atom. In other words,
	\begin{cor}
		$f\in\mathscr{D}(\mathcal{M}(X,\mathcal{A}))$ is self-adjacent in $W\Gamma(\mathcal{M}(X,\mathcal{A}))$ if and only if $Z(f)$ is not an atom.
	\end{cor}
	
	To make $W\Gamma(\mathcal{M}(X,\mathcal{A}))$ a simple graph, we redefine the vertex set of the graph as $\mathscr{D}'(\mathcal{M}(X,\mathcal{A}))= \{f\in\mathscr{D}(\mathcal{M}(X,\mathcal{A})):Z(f)\text{ is an atom}\}$. If $\mu$ is non-atomic, then $\mathcal{A}$ does not contain any atom and hence $W\Gamma(\mathcal{M}(X,\mathcal{A}))$ is an empty graph. For the non-emptiness of $W\Gamma(\mathcal{M}(X,\mathcal{A}))$, we assume that $\mathcal{A}$ contains atleast one atom $A$ with $\mu(X\setminus A)>0$. Then the adjacency relation between two vertices in $\mathscr{D}'(\mathcal{M}(X,\mathcal{A}))$ takes the form as described in the next theorem.
	\begin{thm}\label{Th6.4}
		Let $f,g\in\mathscr{D}'(\mathcal{M}(X,\mathcal{A}))$. Then $f,g$ are adjacent in $W\Gamma(\mathcal{M}(X,\mathcal{A}))$ if and only if $\mu(Z(f)\triangle Z(g))>0$. 
	\end{thm} 
	Consider the equivalence relation $\sim$ on $\mathscr{D}'(\mathcal{M}(X,\mathcal{A}))$ given by: ``$f\sim g$ if and only if $\mu(Z(f)\triangle Z(g))=0$''. For each $f\in \mathscr{D}'(\mathcal{M}(X,\mathcal{A}))$, let $[f]$ denote the equivalence class of $f$ under the relation $\sim$ on $\mathscr{D}'(\mathcal{M}(X,\mathcal{A}))$. From Theorem \ref{Th6.4}, it follows that 
	\begin{enumerate}
		\item for each $f\in\mathscr{D}'(\mathcal{M}(X,\mathcal{A}))$, $[f]$ is a stable set in $W\Gamma(\mathcal{M}(X,\mathcal{A}))$.
		\item if $[f],[g]$ are distinct classes, then for all $f_1\in[f]$ and for all $g_1\in[g]$, $f_1,g_1$ are adjacent in $W\Gamma(\mathcal{M}(X,\mathcal{A}))$.
	\end{enumerate}
	If $W$ is the collection of distinct class representatives under $\sim$ on $\mathscr{D}'(\mathcal{M}(X,\mathcal{A}))$, then we get the following theorem.
	\begin{thm}\label{Th6.5}
		$W\Gamma(\mathcal{M}(X,\mathcal{A}))$ is a complete $|W|$-partite graph. 
	\end{thm}
	It is known that if $A\in\mathcal{A}$ is an atom, then $\mu(X\setminus A)>0$. Moreover, two atoms $A,B\in\mathcal{A}$ are said to be distinct if $\mu(A\triangle B)>0$. So, $|W|=$ number of distinct atoms in $\mathcal{A}$. This observation leads to the following theorem. 
	\begin{thm}
		$W\Gamma(\mathcal{M}(X,\mathcal{A}))$ is a complete bipartite graph if and only if $\mathcal{A}$ contains exactly two distinct atoms.
	\end{thm}
	\begin{thm}
		If $\mathcal{A}$ contains atleast three distinct atoms, then the following hold:
		\begin{enumerate}
			\item $W\Gamma(\mathcal{M}(X,\mathcal{A}))$ is triangulated.
			\item $W\Gamma(\mathcal{M}(X,\mathcal{A}))$ is hypertriangulated.
			\item The girth of $W\Gamma(\mathcal{M}(X,\mathcal{A}))$ is $3$.
			\item No pair of vertices in $\mathscr{D}'(\mathcal{M}(X,\mathcal{A}))$ are orthogonal in $W\Gamma(\mathcal{M}(X,\mathcal{A}))$.
			\item $W\Gamma(\mathcal{M}(X,\mathcal{A}))$ is not a complemented graph.
		\end{enumerate}
	\end{thm}
	We list the values of certain important graph parameters of $W\Gamma(\mathcal{M}(X,\mathcal{A}))$ in the following theorem:
	\begin{thm}
		\hspace*{3cm}
		\begin{enumerate}
			\item The dominating number of  $W\Gamma(\mathcal{M}(X,\mathcal{A}))$ is $2$.
			\item The chromatic number of $W\Gamma(\mathcal{M}(X,\mathcal{A}))$ is $|W|$.
			\item The clique number of $W\Gamma(\mathcal{M}(X,\mathcal{A}))$ is $|W|$. 
		\end{enumerate}
	\end{thm}
	\section{\bf Illustration via two familiar Measure Spaces}\label{Sec7}
	\counterwithin{thm}{subsection}
	\subsection{\bf On the Counting Measure Space}
	Let $X$ be a non-empty set, $\mathcal{A}=\mathscr{P}(X)$, the power set of $X$ and $\mu$, the counting measure on $X$, be defined as follows: $$\mu(E)=\begin{cases}
		|E|\text{ if }E\text{ is finite}\\
		\infty\text{ otherwise}
	\end{cases}$$ 
	Then $\mathcal{M}(X,\mathcal{A})=\mathbb{R}^X$ and $\mathscr{D}(\mathcal{M}(X,\mathcal{A}))=\{f\in\mathcal{M}(X,\mathcal{A})
	:Z(f)\neq\emptyset,X\}$. For the non-emptiness of $\mathscr{D}(\mathcal{M}(X,\mathcal{A}))$, we always consider $|X|>1$. Clearly, each singleton set is an atom in $(X,\mathcal{A},\mu)$. \\
	
	Choosing $(X, \mathcal{A}, \mu)$ as the counting measure space, the following two theorems list the important features of the comaximal graph and the annihilator graph of $\mathcal{M}(X, \mathcal{A})$.
	\begin{thm}\label{Th7.1.1}
		Consider the comaximal graph $\Gamma'_2(\mathcal{M}(X,\mathcal{A}))$ of $\mathcal{M}(X, \mathcal{A})$ and $f,g\in\mathscr{D}(\mathcal{M}(X,\mathcal{A}))$.
		\begin{enumerate}
			\item There exists a vertex adjacent to both $f$ and $g$ if and only if $Z(f)\cup Z(g)\neq X$.
			\item $d(f,g)=\begin{cases}
				1&Z(f)\cap Z(g)= \emptyset\\
				2&Z(f)\cap Z(g)\neq \emptyset\text{ and }Z(f)\cup Z(g)\neq X\\
				3&Z(f)\cap Z(g)\neq \emptyset\text{ and }Z(f)\cup Z(g)=X
			\end{cases}$
			\item $ecc(f)=\begin{cases}
				2&\text{ if }|Z(f)|=1\\
				3&\text{ otherwise}
			\end{cases}$\label{7.1.1.3}
			\item $\Gamma'_2(\mathcal{M}(X,\mathcal{A}))$ is complete bipartite if and only if $|X|=2$.
			\item The diameter of $\Gamma'_2(\mathcal{M}(X,\mathcal{A}))$ is $\begin{cases}
				2&|X|=2\\
				3&|X|\geq 3
			\end{cases}$
			\item The girth of $\Gamma'_2(\mathcal{M}(X,\mathcal{A}))$ is $\begin{cases}
				4&|X|=2\\
				3&|X|\geq 3
			\end{cases}$
			\item $f$ is a vertex of a triangle if and only if $|X\setminus Z(f)|>1$.
			\item $\Gamma'_2(\mathcal{M}(X,\mathcal{A}))$ is never triangulated and never hypertriangulated.
			\item $f\perp g$ if and only if $Z(f)=X\setminus Z(g)$.
		\end{enumerate}
	\end{thm}
	\begin{thm}
		Consider the annihilator graph $AG(\mathcal{M}(X,\mathcal{A}))$ of $\mathcal{M}(X,\mathcal{A})$ and $f,g\in\mathscr{D}(\mathcal{M}(X,\mathcal{A}))$.
		\begin{enumerate}
			\item There exists a vertex adjacent to both $f,g$ if and only if either $Z(f)\cup Z(g)\neq X$ or $Z(f)\cap Z(g)\neq\emptyset$ or $|Z(f)|,|Z(g)|\geq 2$.
			\item $d(f,g)=\begin{cases}
				1&Z(f)\setminus Z(g)\neq\emptyset\text{ and }Z(g)\setminus Z(f) \neq\emptyset\\
				2&\text{ otherwise}
			\end{cases}$
			\item $ecc(f)=2$
			\item $AG(\mathcal{M}(X,\mathcal{A}))$ is complete bipartite if and only if $|X|=2$.
			\item $\Gamma'_2(\mathcal{M}(X,\mathcal{A}))= AG(\mathcal{M}(X,\mathcal{A}))$ if and only if $|X|=2$ if and only if $\Gamma(\mathcal{M}(X,\mathcal{A})) = AG(\mathcal{M}(X,\mathcal{A}))$.
			\item The girth of $AG(\mathcal{M}(X,\mathcal{A}))$ is $\begin{cases}
				4&|X|=2\\
				3&|X|\geq 3
			\end{cases}$
			\item $AG(\mathcal{M}(X,\mathcal{A}))$ is not hypertriangulated.
			\item $f$ has an orthogonal complement in $AG(\mathcal{M}(X,\mathcal{A}))$ if and only if either $|Z(f)|=1$ or $|X\setminus Z(f)|=1$.
			\item $AG(\mathcal{M}(X,\mathcal{A}))$ is uniquely complemented if and only if $|X|\leq 3$.
		\end{enumerate}
	\end{thm}
	For this measure space, the Theorem~\ref{Th6.5} takes the following form:
	\begin{thm}
		$W\Gamma(\mathcal{M}(X,\mathcal{A}))$ is a complete $|X|$-partite graph where the stable sets are $W_x=\{f\in\mathcal{M}(X,\mathcal{A}):Z(f)=\{x\}\}$, $x\in X$.
	\end{thm}
	In the counting measure space, the equivalence relation described before, reads as ``$f\sim g$ if and only if $Z(f)=Z(g)$'' and therefore $[f]=\{g\in\mathcal{M}(X,\mathcal{A}):Z(f)=Z(g)\}$. Hence, the vertex set of the induced graph $G_2$ of $\Gamma'_2(\mathcal{M}(X,\mathcal{A}))$ becomes $V=\{1_A:A\neq\emptyset,X\}$. In $G_2$, two distinct vertices $1_A,1_B$ are adjacent if and only if $X\setminus A\cap X\setminus B=\emptyset$. \\
	
	In the next few results we observe that in the counting measure space, the chromatic number and clique number of $\Gamma'_2(\mathcal{M}(X,\mathcal{A}))$ solely depend on the cardinality of the underlying space $X$. We prove these results using the induced graph $G_2$ of $\Gamma'_2(\mathcal{M}(X,\mathcal{A}))$ and Theorem~\ref{Th4.5} and Theorem~\ref{Th4.3} of Section \ref{Sec3}.
	\begin{thm}\label{Th7.1}
		The chromatic number of $G_2$ is $|X|$.
	\end{thm}
	\begin{proof}
		Clearly, $\{1_{X\setminus\{x\}}:x\in X\}$ is a complete subgraph of $G_2$. Therefore, $\chi(G_2)\geq|X|$. We begin with coloring each $1_{X\setminus\{x\}}\in V$ by using distinct colors. Let $1_A\in V\setminus\{1_{X\setminus\{x\}}:x\in X\}$. Then $|X\setminus A|\geq 2$. We color $1_A$ by one of the colors of $1_{X\setminus\{x\}}$, where $x\in X\setminus A$. Suppose $1_A,1_B\in V$ are colored by the same color, say by the color of $1_{X\setminus\{x\}}$. Then by hypothesis, $x\in X\setminus A\cap X\setminus B$. Consequently, $1_A,1_B$ are not adjacent in $G_2$. It gives a consistent coloring of $G_2$ by $|X|$-many colors. Therefore, $\chi(G_2)\leq|X|$.
	\end{proof}
	\begin{cor}
		The chromatic number of $\Gamma'_2(\mathcal{M}(X,\mathcal{A}))$ is $|X|$.
	\end{cor}
	\begin{thm}
		The clique number of $G_2$ is $|X|$.
	\end{thm}
	\begin{proof}
		Since $\{1_{X\setminus\{x\}}:x\in X\}$ is a complete subgraph of $G$, $cl(G_2)\geq|X|$. For any graph $H$, we know that $cl(H)\leq\chi(H)$. Hence by Theorem \ref{Th7.1}, $cl(G_2)\leq|X|$.
	\end{proof}
	\begin{cor}
		The clique number of $\Gamma'_2(\mathcal{M}(X,\mathcal{A}))$ is $|X|$. 
	\end{cor}

	\begin{rem}
		A graph is said to be weakly perfect if the clique number and the chromatic number are equal. Since $cl(\Gamma'_2(\mathcal{M}(X,\mathcal{A})))=|X|=\chi(\Gamma'_2(\mathcal{M}(X,\mathcal{A})))$, in this measure space, $\Gamma'_2(\mathcal{M}(X,\mathcal{A}))$ is a weakly perfect graph.
	\end{rem}
	We now compare the comaximal graph and the zero-divisor graph of $\mathcal{M}(X,\mathcal{A})$ in the counting measure space. For this, we need the following lemmas.
	\begin{lem}\label{Lem7.1.1}
		For each $f\in\mathcal{M}(X,\mathcal{A})$, $|[f]|=|\mathbb{R}^ {|X\setminus Z(f)|}|$.
	\end{lem}
	\begin{proof}
		Let $X\setminus Z(f)=\{x_\alpha:\alpha\in\Lambda\}$. Define $\eta:[f]\to(\mathbb{R}\setminus\{0\})^{|\Lambda|}$ by $\eta(g)=\Pi_{\alpha\in\Lambda}(g(x_\alpha))$. $\eta$ is well defined, since for all $g\in[f]$, $X\setminus Z(g)=X\setminus Z(f)$. Let $g_1,g_2\in[f]$ be such that $\eta(g_1)=\eta(g_2)$; i.e., $g_1(x_\alpha)= g_2(x_\alpha)$ for all $\alpha\in\Lambda$. Also for all $x\notin X\setminus Z(f)$, $g_1(x)=0=g_2(x)\implies  g_1=g_2$ on $X$. Therefore $\eta$ is injective. Let $y\in(\mathbb{R}\setminus\{0\})^{|\Lambda|}$. Then $y$ can be written as $y=(y_\alpha)_{\alpha\in\Lambda}$, where $y_\alpha\in\mathbb{R}\setminus\{0\}$ for all $\alpha\in\Lambda$. Let $g(x)=\begin{cases}
			y_\alpha&\text{ if }x=x_\alpha\\
			0&\text{ otherwise}
		\end{cases}$. Since $y_\alpha\neq 0$ for all $\alpha\in\Lambda$, $X\setminus Z(g)=\{x_\alpha:\alpha\in\Lambda\}=X\setminus Z(f)\implies g\in[f]$. Clearly $\eta(g)=y$. Consequently $\eta$ is a bijection i.e., $|[f]|= |(\mathbb{R}\setminus\{0\})^{|\Lambda|}|=|\mathbb{R}^ {|X\setminus Z(f)|}|$.
	\end{proof}
	Let $f\in\mathscr{D}(\mathcal{M}(X,\mathcal{A}))$. By Theorem \ref{Th7.1.1}(\ref{7.1.1.3}) it follows that, $ecc(f)=2$ in $\Gamma'_2(\mathcal{M}(X,\mathcal{A}))$ if and only if $Z(f)$ is a singleton set. Similarly, it can be proved that, $ecc(f)=2$ in $\Gamma(\mathcal{M}(X,\mathcal{A}))$ if and only if $X\setminus Z(f)$ is a singleton set. These observations lead to the following important lemma.
	\begin{lem}\label{Lem7.1.2}
		If $\Gamma(\mathcal{M}(X,\mathcal{A}))$ and $\Gamma'_2(\mathcal{M}(X,\mathcal{A}))$ are graph isomorphic, then for each $x\in X$ there exists $y\in X$ such that $|[1_x]|=|[1_{X\setminus\{y\}}]|$.
	\end{lem}
	\begin{proof}
		Let $\psi:\Gamma\to\Gamma'_2$ be a graph isomorphism. Fix an $x\in X$. Let $f\in[1_x]$. Then $X\setminus Z(f)=\{x\}\implies ecc(f)=2$ in $\Gamma$. Since $\psi$ is a graph isomorphism, $ecc(\psi(f))=2$ in $\Gamma'_2\implies Z(\psi(f))$ is a singleton set i.e., $Z(\psi(f))=\{y\}$ for some $y\in X$. Therefore, $\psi(f)\in[1_{X\setminus\{y\}}]$. Conversely let $g\in[1_{X\setminus\{y\}}]$. Then $Z(g)=\{y\}\implies ecc(g)=2$ in $\Gamma'_2\implies ecc(\psi^{-1}(g))=2$ in $\Gamma\implies X\setminus Z(\psi^{-1}(g))$ is a singleton set, i.e., $X\setminus Z(\psi^{-1}(g))=\{x'\}$ for some $x'\in X$. If $x\neq x'$, then $1_x,\psi^{-1}(g)$ are adjacent in $\Gamma$. Therefore their images under $\psi$ are adjacent in $\Gamma'_2$, i.e., $\psi(1_x),g$ are adjacent in $\Gamma'_2$. But $1_x\in [1_x]\implies \psi(1_x)\in[1_{X\setminus\{y\}}]$, by our definition. Therefore, $Z(\psi(1_x))=\{y\}=Z(g)\implies Z(\psi(1_x))\cap Z(g)=\{y\}$. So, $\psi(1_x),g$ are not adjacent in $\Gamma'_2$, which is a contradiction. Therefore, $X\setminus Z(\psi^{-1}(g))=\{x\}\implies\psi^{-1}(g)\in[1_x]$. Hence, the restriction map $\psi:[1_x]\to[1_{X\setminus\{y\}}]$ is a bijection, proving $|[1_x]|=|[1_{X\setminus\{y\}}]|$. 
	\end{proof}
	\begin{thm}\label{Th7.16}
		$\Gamma(\mathcal{M}(X,\mathcal{A}))$ and $\Gamma'_2(\mathcal{M}(X,\mathcal{A}))$ are graph isomorphic if and only if $X$ is atmost countable.
	\end{thm}
	\begin{proof}
		If $X$ is atmost countable, then for any $f\in\mathscr{D}$, $|X\setminus Z(f)|$ is countable and hence by Lemma \ref{Lem7.1.1}, $|[f]|= |\mathbb{R}^{|X\setminus Z(f)|}|=|\mathbb{R}| =\mathfrak{c}$. So, $|[f]|=|[g]|$ for all $f,g\in\mathscr{D}$. Therefore by Theorem \ref{Iso}, $\Gamma$ and $\Gamma'_2$ are isomorphic. Conversely let $X$ be uncountable. Then for any $x,y\in X$, $X\setminus\{x\}$ is uncountable $\implies|\mathbb{R}^{|X\setminus \{x\}}||>|\mathbb{R}^{|\{y\}|}|$ ; i.e., $|[1_{X\setminus\{x\}}]|> |[1_y]|$. By Lemma \ref{Lem7.1.2}, $\Gamma$ and $\Gamma'_2$ are not graph isomorphic.
	\end{proof}
	\begin{rem}\label{Rem7.17}
		Let $(X,\mathcal{A},\mu)$ be the counting measure space, where $X$ is an uncountable set. Then by Theorem \ref{Th7.16}, the zero-divisor graph and the comaximal graph of $\mathcal{M}(X, \mathcal{A})$ are never isomorphic as graphs. But $G$ and $G_2$ are graph isomorphic.
	\end{rem}
	\subsection{\bf On the Lebesgue Measure Space}
	Let $(\mathbb{R},\mathcal{A},\ell)$ be the Lebesgue measure space on $\mathbb{R}$ and $\mathcal{A}$, the collection of all Lebesgue measurable sets on $\mathbb{R}$. We know that, the Lebesgue measure on $\mathbb{R}$ is a non-atomic measure. So we have the following results:
	\begin{thm}
		\hspace{3cm}
		\begin{enumerate}
			\item The diameter of the comaximal graph $\Gamma'_2(\mathcal{M}(\mathbb{R},\mathcal{A}))$ of $\mathcal{M}(\mathbb{R},\mathcal{A})$ is $3$.
			\item The eccentricity of every vertex in $\Gamma'_2(\mathcal{M}(\mathbb{R},\mathcal{A}))$ is $3$.
			\item The girth of $\Gamma'_2(\mathcal{M}(\mathbb{R},\mathcal{A}))$ is $3$.
			\item $\Gamma'_2(\mathcal{M}(\mathbb{R},\mathcal{A}))$ is always a triangulated graph.
			\item The annihilator graph $AG(\mathcal{M}(\mathbb{R},\mathcal{A}))$ of $\mathcal{M}(\mathbb{R},\mathcal{A})$ is always triangulated as well as hypertriangulated.
			\item $AG(\mathcal{M}(\mathbb{R},\mathcal{A}))$ can never be a complemented graph.
			\item The zero-divisor graph and the comaximal graph of $\mathcal{M}(\mathbb{R},\mathcal{A})$ are not isomorphic to the annihilator graph of $\mathcal{M}(\mathbb{R},\mathcal{A})$.
			\item The weakly zero-divisor graph of $\mathcal{M}(\mathbb{R},\mathcal{A})$ is an empty graph.
		\end{enumerate} 
	\end{thm}
	As $\mathcal{A}$ contains $2^\mathfrak{c}$-many distinct Lebesgue measurable sets, $|\mathcal{M}(\mathbb{R},\mathcal{A})|= 2^\mathfrak{c}$. We now show that the comaximal graph and the zero-divisor graph of $\mathcal{M}(\mathbb{R},\mathcal{A})$ are isomorphic.
	\begin{lem}\label{Lem7.2.1}
		Let $A\in\mathcal{A}$ be such that $A$ is compact in $\mathbb{R}$ (with respect to the usual topology). Then there exist $a,b\in A$ such that $\ell(A\cap [a,b])=\ell(A)$.
	\end{lem}
	\begin{proof}
		Since $A$ is compact in $\mathbb{R}$, $A$ is closed and $A\subset [x,y]$ for some $x,y\in\mathbb{R}$. Clearly, $\ell(A)\leq y-x$, finite. Let $B=\{z\in[a,b]:\ell(A\cap [z,y])=\ell(A)\}$ and  $a=\sup B$. Since $x\in B$, $B\neq\emptyset$ and hence such $a$ exists and $x\leq a<y$. If possible let $a\notin A$, then $a$ is not a limit point of $A$. There exists $\delta>0$ such that $(a-\delta,a+\delta)\cap A=\emptyset$. Therefore $\ell((a-\delta,a+\delta)\cap A)=0$. Since $a$ is the supremum of the set $B$, there exists $a'\in B$ such that $a-\delta<a'<a$. Now $a'\in B\implies\ell(A\cap [a',y])=\ell(A)$. Therefore $\ell(A\cap [x,a'))=0$. This along with the fact $\ell((a-\delta,a+\delta)\cap A)=0$ imply $\ell(A\cap[x,a+\delta))=0\implies\ell(A\cap [a+\delta,y])=\ell(A)$ i.e., $a+\delta\in B$, which contradicts that $a=\sup B$. Therefore $a\in A$. Also $[x,a)=\bigcup\limits_{n=1}^\infty [x,a-\frac{1}{n}]\implies\ell(A\cap[x,a))\leq\sum\limits_{n=1}^{\infty}\ell(A\cap [x,a-\frac{1}{n}])$. Since $a=\sup B$, $\ell(A\cap [x,a-\frac{1}{n}])=0$ for all $n$ and hence $\ell(A\cap[x,a))=0\implies\ell(A\cap[a,y])=\ell(A)$. Similarly if $b=\inf \{z\in[a,y]:\ell(A\cap[a,z])=\ell(A)\}$, then $b\in A$ and $\ell(A\cap[a,b])=\ell(A)$. 
	\end{proof}
	\begin{lem}\label{Lem7.2.2}
		Let $A\in\mathcal{A}$ and $\ell(A)>0$. Then for every $r\in[0,\ell(A)]$, there exists $A_r\in\mathcal{A}$ such that $A_r\subset A$ and $\ell(A_r)=r$.
	\end{lem}
	\begin{proof}
		We consider the function $\psi:\mathbb{R}\to[0,\ell(A)]$ given by $\psi(x)=\ell(A\cap(-\infty,x])$. Since $|\psi(x)-\psi(y)|\leq |x-y|$ for all $x,y\in \mathbb{R}$, $\psi$ is continuous on $\mathbb{R}$. Also $\lim\limits_{x\to -\infty}\psi(x)=0$ and $\lim\limits_{x\to \infty}\psi(x)=\ell(A)$. Therefore $\psi$ takes every value on $[0,\ell(A)]$. i.e., For each $r\in[0,\ell(A)]$, there exists $x_r\in\mathbb{R}$ such that $\psi(x_r)=r$. So, $\ell(A\cap(-\infty,x_r])=r$.
	\end{proof}
	
	In the Lemma \ref{Lem7.2.2}, if $A$ is a compact set, then for some $a,b\in A$, $\psi : [a,b] \rightarrow [0,\ell(A)]$ may be taken as $\psi(x)=\ell(A\cap[a,x])$. Also for each $r\in [0,\ell(A)]$, there exists $y_r\in A\cap [a,b]$ such that $\ell(A\cap[a,y_r])=r$.
	
	\begin{thm}\label{Th7.2.4}
		If $\ell(A)>0$ for some $A\in\mathcal{A}$, then $A$ contains an uncountable set of measure zero.
	\end{thm}
	\begin{proof}
		Without loss of generality let $\ell(A\cap [0,1])>0$. By the regularity of $\ell$, there exists a compact set $K$ in $\mathbb{R}$ such that $\ell(K)>0$ and $K\subset A\cap [0,1]$. By Lemma \ref{Lem7.2.1}, there exists $a,b\in K\cap [0,1]$ such that $\ell(K\cap[a,b])=\ell(K)$. By the observation after Lemma~\ref{Lem7.2.2}, there exist $x_{11},x_{12}\in K\cap [a,b]$ such that $\ell(K\cap[a,x_{11}])=\frac{1}{3}\ell(K)$ and $\ell(K\cap[a,x_{12}])=\frac{2}{3}\ell(K)$. Let $I_{11}=K\cap [a,x_{11}]$, $I_{12}=K\cap[x_{11},x_{12}]$ and $I_{13}=K\cap[x_{12},b]$. Then $\ell(I_{11})=\frac{1}{3}\ell(K)=\ell(I_{12})=\ell(I_{13})$; i.e., we trisect $K\cap[a,b]$ into three parts $I_{11},I_{12},I_{13}$ with equal measure. Let $E_1=I_{11}\sqcup I_{13}$. Then $E_1\in\mathcal{A}$ and $\ell(E_1)=\frac{2}{3} \ell(K)$. Similarly we trisect $I_{11}$ into $I_{21}=K\cap[a,x_{21}], I_{22}=K\cap[x_{21},x_{22}]$ and $I_{23}=K\cap[x_{22},x_{11}]$ with equal measure, where $x_{21},x_{22}\in K\cap [a,b]$. We also trisect $I_{13}$ into $I_{24}=K\cap[x_{12},x_{23}], I_{25}=K\cap[x_{23},x_{24}]$ and $I_{26}=K\cap[x_{24},b]$ with equal measure, where $x_{23},x_{24}\in K\cap [a,b]$. Let $E_2=(I_{21}\sqcup I_{23})\sqcup(I_{24}\sqcup I_{26})$. Then $E_2\in\mathcal{A}$ and $\ell(E_2)= (\frac{2}{3})^4\ell(K)$. Continuing this process, at the $n$-th stage, we get $E_n\in\mathcal{A}$ which is the union of $2^n$-many disjoint measurable sets $I_{n1},\ldots, I_{n2^n}$ each of which has measure $\frac{1}{3^n}\ell(K)$. Thus $\ell(E_n)=(\frac{2}{3})^{2^n}\ell(K)$. Let $E=\bigcap\limits_{n=1}^\infty E_n$. Then $E\in\mathcal{A}$ and $\ell(E)=\lim\limits_{n\to\infty}\ell(E_n)=\ell(K).\lim\limits_{n\to\infty}(\frac{2}{3})^{2^n}=0$. For every chain of Lebesgue measurable sets of the form $I_{mn}$, $x_{mn}\in\bigcap I_{mn}$ and $\bigcap I_{mn}\subset\bigcap\limits_{n=1}^\infty E_{n}\subset E\implies E \neq\emptyset$. We claim that $E$ is uncountable. If possible let $E$ be countable, say $E=\{x_n:n\in\mathbb{N} \}$. Since $x_1\in E$, $x_1\in E_1\implies$ there exists $k\in\{1,3\}$ such that $x_1\notin I_{1k_1}$. Similarly $x_2\in E_2\implies x_2\notin I_{2k_2}$ where $I_{2k_2}$ is one of the measurable sets in $E_2$ such that $I_{2k_2}\subset I_{1k_1}$. Thus for each $n\in\mathbb{N}$, there exists some $k_n\in\{1,2,...,2^n\}$ such that $x_n\notin I_{nk_n}$ and $I_{nk_n}\subset I_{(n-1)k_{n-1}}$. Since $I_{nk_n}\subset I_{(n-1)k_{n-1}}$ for all $n\in\mathbb{N}$, $\bigcap\limits_{n=1}^\infty I_{nk_n}\neq\emptyset$, say $x\in\bigcap\limits_{n=1}^\infty I_{nk_n}$. Then $x\in E$ and $x\notin\{x_1,x_2,...,x_n,....\}=E $, a contradiction. Hence, $E$ is an uncountable set.
	\end{proof}
	\begin{cor}\label{Cor7.2.5}
		If $\ell(A)>0$ for some $A\in\mathcal{A}$, then $A$ contains exactly $2^\mathfrak{c}$-many Lebesgue measurable sets.
	\end{cor}
	\begin{proof}
		Clearly, $A$ cannot contain more than $2^\mathfrak{c}$-many sets, as $A\subset\mathbb{R}$. By Theorem \ref{Th7.2.4}, $A$ contains an uncountable set of measure zero, say $E$. Since $\ell$ is a complete measure, and $E$ is uncountable, $E$ contains $2^\mathfrak{c}$-many Lebesgue measurable sets. Consequently, $A$ contains $2^\mathfrak{c}$-many Lebesgue measurable sets.
	\end{proof}
	We already made a partition of $\mathscr{D}(\mathcal{M}(\mathbb{R},\mathcal{A}))$ by the equivalence relation ``$\sim$'' given by: ``$f\sim g$ if and only if $\ell(Z(f)\triangle Z(g))=0$''. For each $f\in\mathscr{D}(\mathcal{M}(\mathbb{R},\mathcal{A})) $, we consider $[f]$ as the equivalence class of $f$ under the equivalence relation ``$\sim$'' which is precisely the set $\{g\in\mathscr{D}(\mathcal{M}(\mathbb{R},\mathcal{A})):\ell(Z(f)\triangle Z(g))=0\}$. Clearly, for each $f\in\mathscr{D}(\mathcal{M}(\mathbb{R},\mathcal{A}))$, $|[f]|\leq|\mathcal{M}(\mathbb{R},\mathcal{A})|=2^\mathfrak{c}$. $V$ is the collection of all class representatives of the form $1_A$ under $\sim$, as in all the previous cases.
	\begin{lem}\label{Lem7.2.6}
		For each $f\in\mathscr{D}(\mathcal{M}(\mathbb{R},\mathcal{A}))$, $|[f]|=2^\mathfrak{c}$.
	\end{lem}
	\begin{proof}
		Let $f\in\mathscr{D}(\mathcal{M}(\mathbb{R},\mathcal{A}))$. Then $\ell(X\setminus Z(f))>0$. By Corollary \ref{Cor7.2.5}, $X\setminus Z(f)$ contains exactly $2^\mathfrak{c}$-many Lebesgue measurable sets. Let $\{A_\alpha:\alpha\in\Lambda\}$ be the collection of all Lebesgue measurable subsets of $X\setminus Z(f)$, where $|\Lambda|=2^\mathfrak{c}$. For each $\alpha\in\Lambda$, let $f_\alpha:\mathbb{R}\to\mathbb{R}$ be defined by
		$$f_\alpha(x)=\begin{cases}
			0&\text{ if }x\in Z(f)\\
			1&\text{ if }x\in A_\alpha\\
			2&\text{ otherwise}
		\end{cases}$$
		Clearly $f_\alpha\in[f]$ for all $\alpha\in\Lambda$. Also for distinct $\alpha,\beta\in\Lambda$, $f_\alpha\neq f_\beta$. Therefore $|[f]|\geq 2^\mathfrak{c}$. Hence $|[f]|=2^\mathfrak{c}$.
	\end{proof}
	\begin{thm}\label{Th7.2.7}
		In the Lebesgue measure space $(\mathbb{R},\mathcal{A},\ell)$, the zero-divisor graph and the comaximal graph on $\mathcal{M}(\mathbb{R}, \mathcal{A})$ are isomorphic.
	\end{thm}
	\begin{proof}
		By Lemma \ref{Lem7.2.6}, $|[f]|=2^\mathfrak{c}$ for all $f\in\mathscr{D}(\mathcal{M}(\mathbb{R},\mathcal{A}))$. In particular $|[1_A]|=|[1_{X\setminus A}]|$ for each $1_A\in V$. Hence by Theorem \ref{Iso}, the zero-divisor graph and the comaximal graph on $\mathcal{M}(\mathbb{R},\mathcal{A})$ are isomorphic.
	\end{proof}
	Since the Lebesgue measure on $\mathbb{R}$ is a non-atomic measure, we fell tempted to ask the following question.
	\begin{qs}
		Suppose $(X,\mathcal{A},\mu)$ is a non-atomic measure space. Are the zero-divisor graph and the comaximal graph of $\mathcal{M}(X,\mathcal{A})$ isomorphic?
	\end{qs}
		
\end{document}